\documentclass[]{article}

\usepackage{amsmath,amsthm,amssymb,mathtools, bm}

\usepackage{graphicx}

\usepackage{algorithm}
\usepackage{algpseudocode}
\allowdisplaybreaks

\newcommand{\sgn}{\text{sgn}}
\newcommand{\outerprod}{\ensuremath{\mathbin{\text{\raisebox{1.2pt}{$\scriptscriptstyle{\bm\otimes}$}}}}} 

\newcommand{\Diag}[1]{\textup{Diag}{(#1)}}

\newcommand{\condWm}{$\text{\textup{(W{\scriptsize m}}}$$\text{\textup{)}}$\ }

\newcommand\Algphase[1]{%
\vspace*{-.4\baselineskip}\Statex\hspace*{\dimexpr-\algorithmicindent-2pt\relax}\rule{\textwidth}{0.4pt}%
\Statex\hspace*{-\algorithmicindent}{#1}%
\vspace*{-.7\baselineskip}\Statex\hspace*{\dimexpr-\algorithmicindent-2pt\relax}\rule{\textwidth}{0.4pt}%
}

\renewcommand{\algorithmicrequire}{\textbf{Input:}}
\renewcommand{\algorithmicensure}{\textbf{Output:}}
\renewcommand{\algorithmicrepeat}{}
\renewcommand{\algorithmicuntil}{}

\newtheorem{theorem}{Theorem}
\newtheorem{example}[theorem]{Example}
\newtheorem{definition}[theorem]{Definition}
\newtheorem{proposition}[theorem]{Proposition}
\newtheorem{lemma}[theorem]{Lemma}

\title{Canonical polyadic decomposition of third-order tensors: relaxed uniqueness conditions and algebraic algorithm\footnote{Research supported by: (1) Research Council KU Leuven: C1 project c16/15/059-nD, CoE PFV/10/002 (OPTEC), PDM postdoc grant; (2) F.W.O.:  project  G.0830.14N, G.0881.14N;  (3) the Belgian Federal Science Policy Office: IUAP P7 (DYSCO II,  Dynamical systems, control
 and optimization,  2012-2017); (4)  EU: The research leading to these results has received funding from the European Research Council under the European Union's Seventh Framework Programme (FP7/2007-2013) / ERC Advanced Grant: BIOTENSORS (no.  339804). This paper reflects only the authors' views and the Union is not liable for any use that may be made of the contained information.}}

\author{Ignat Domanov and Lieven De Lathauwer}

\begin{document}

\maketitle

\begin{abstract}
Canonical Polyadic Decomposition (CPD) of a third-order tensor is a minimal decomposition
into a sum of rank-$1$ tensors.  
We find new mild deterministic conditions for the uniqueness of individual rank-$1$ tensors in  CPD  and 
  present an  algorithm to recover them.  We call the   algorithm   ``algebraic'' because
  it  relies only on standard  linear algebra. It does not involve more advanced procedures than the computation of the null space of a matrix and  eigen/singular value decomposition. 
    Simulations  indicate that the new conditions for uniqueness  and the working assumptions for the algorithm hold for a randomly generated  $I\times J\times K$ tensor of rank $R\geq K\geq J\geq I\geq 2$ if $R$ is  bounded as  $R\leq (I+J+K-2)/2 + (K-\sqrt{(I-J)^2+4K})/2$ at least for the dimensions that we have tested.
  This improves upon the  famous Kruskal bound for uniqueness $R\leq (I+J+K-2)/2$ as soon as $I\geq 3$.

	In the particular case $R=K$, the new bound above is equivalent to the bound $R\leq(I-1)(J-1)$ which is known to be necessary and sufficient
	for the generic uniqueness of the CPD. An existing algebraic algorithm (based on simultaneous diagonalization of a set of matrices) 
	  computes the CPD  under the more restrictive constraint $R(R-1)\leq I(I-1)J(J-1)/2$ (implying that $R<(J-\frac{1}{2})(I-\frac{1}{2})/\sqrt{2}+1$). On the other hand,
	  optimization-based algorithms fail to compute the CPD in a reasonable amount of time even in the low-dimensional case $I=3$, $J=7$, $K=R=12$.
	  By comparison, in our approach the computation  takes  less than $1$ sec.
	  We demonstrate  that, at least for $R\leq 24$,  our algorithm can  recover the rank-$1$ tensors in the CPD up to  $R\leq(I-1)(J-1)$.
\end{abstract}


{\bf AMS subject classifacation:} 15A69,  15A23


\section{Introduction}
Let $\mathbb F$ denote the field of real or complex numbers
and $\mathcal T\in\mathbb F^{I\times J\times K}$ denote a third-order tensor with entries $t_{ijk}$. 
By definition,  $\mathcal T$ is {\em rank-$1$} if it equals the outer product of three nonzero vectors  $\mathbf a\in\mathbb F^I$, $\mathbf b\in\mathbb F^J$, and  $\mathbf c\in\mathbb F^K$: $\mathcal T=\mathbf a\outerprod\mathbf b\outerprod\mathbf c$, which means that  $t_{ijk}=a_ib_jc_k$ for all values of indices. 

 A {\em Polyadic  Decomposition} of  $\mathcal T$ expresses $\mathcal T$ as a  sum of rank-$1$ terms:
\begin{equation}
\mathcal T=\sum\limits_{r=1}^R\mathbf a_r\outerprod \mathbf b_r\outerprod \mathbf c_r,\qquad
 \left(\text{or}\ t_{ijk} = \sum\limits_{r=1}^Ra_{ir}b_{jr}c_{kr}\right)
 \label{eqintro2}
\end{equation}
where 
$$
\mathbf a_r = [a_{1r}\,\dots\,a_{Ir}]^T \in \mathbb F^{I},\ 
\mathbf b_r = [b_{1r}\,\dots\,b_{Jr}]^T \in \mathbb F^{J},\ 
\mathbf c_r = [c_{1r}\,\dots\,c_{Kr}]^T \in \mathbb F^{K}.
$$
If the number $R$ of rank-1 terms in \eqref{eqintro2} is minimal, then
\eqref{eqintro2} is called the {\em Canonical Polyadic  Decomposition} (CPD) of  $\mathcal T$ and
$R$  is called the {\em rank }of  $\mathcal T$ (denoted by $r_{\mathcal T}$).
It is clear that in (\ref{eqintro2}) the rank-1 terms can be arbitrarily permuted and that vectors within the same rank-1 term can be arbitrarily scaled provided the overall rank-1 term remains the same. The CPD of a tensor {\em is unique } when it is only subject to these trivial indeterminacies.

We write \eqref{eqintro2} as $\mathcal T=[\mathbf A,\mathbf B,\mathbf C]_R$, where
the matrices
$\mathbf A :=\left[\begin{matrix}\mathbf a_1&\dots&\mathbf a_R\end{matrix}\right] \in\mathbb F^{I\times R}$, $\mathbf B :=\left[\begin{matrix}\mathbf b_1&\dots&\mathbf b_R\end{matrix}\right]\in\mathbb F^{J\times R}$ and $\mathbf C :=\left[\begin{matrix}\mathbf c_1&\dots&\mathbf c_R\end{matrix}\right]\in\mathbb F^{K\times R}$
are called the  {\em first}, {\em second} and {\em third factor matrix} of $\mathcal T$, respectively.
It may happen that the CPD of a tensor $\mathcal T$ is not unique but that nevertheless,  for any  two CPDs  $\mathcal T=[\mathbf A,\mathbf B,\mathbf C]_R$ and $\mathcal T=[\bar{\mathbf A},\bar{\mathbf B},\bar{\mathbf C}]_R$, 
the factor matrices in a certain mode, say the matrices $\mathbf C$ and $\bar{\mathbf C}$,  coincide up  to column permutation and scaling.
We say that  {\em the third  factor matrix of $\mathcal T$ is unique}. For instance, it is well known that if  two or more columns of the third factor matrix of $\mathcal T$ have collinear vectors, then the CPD is not unique. Nevertheless, the third factor matrix can still be unique \cite[Example 4.11]{PartI}. 

 The literature shows some variation in terminology.
 The CPD was introduced by F.L. Hitchcock in \cite{Hitchcock} 
 and was later referred to
 as Canonical Decomposition (CANDECOMP) \cite{1970_Carroll_Chang}, Parallel Factor Model (PARAFAC)
 \cite{Harshman1970,1994HarshmanLundy}, and Topographic Components Model \cite{1988Topographic}. 
Uniqueness of one factor matrix is called 
{\em uni-mode uniqueness} in \cite{GuoMironBrieStegeman, Zhang20131918}.
Uniqueness of the CPD is often called {\em essential uniqueness} in engineering papers and  {\em specific identifiability} in algebraic geometry papers.
It is its uniqueness properties that make  CPD a basic tool for signal
separation and data analysis, with many concrete applications in telecommunication,
array processing, machine learning, etc. \cite{LievenCichocki2013, ComoJ10,KoldaReview,Lieven-Nikos_overview}.

The contribution of this paper is twofold. First, we find very mild conditions for uniqueness of CPD and, second, we provide 
an  {\em algebraic } algorithm  for its computation, i.e. an algorithm that recovers the CPD from $\mathcal T$ by means of conventional linear algebra (basically by  taking the orthogonal  complement of a subspace and  computing generalized  eigenvalue decomposition (GEVD)).

Algebraic algorithms are important from a computational point view in the following sense.
In practice,  the factor matrices of $\mathcal T$ are most often obtained  as 
the solution of the  optimization problem
\begin{equation*}
\min \|\widehat{\mathcal T}-[\mathbf A,\mathbf B,\mathbf C]_R\|,\qquad \text{ s.t. }\quad
 \mathbf A\in\mathbb F^{I\times R},\ \mathbf B\in\mathbb F^{J\times R},\ \mathbf C\in\mathbb F^{K\times R},
\end{equation*}
 where $\|\cdot\|$ denotes  a suitable  norm \cite{Sorber}.   
 The limitations of this approach are not very well-known. Algebraic algorithms may provide a good initial guess.
 In Example \ref{example:manyinits} we illustrate that even in a small-scale problem such as the CPD of a rank-$12$ tensor of dimensions $3\times 7\times 12$, the optimization approach may require many initializations and iterations,  although the solution can be computed algebraically without a problem.

{\em Basic notation and conventions.} 
Throughout the paper  $C_n^k$ denotes the binomial coefficient,  
$$
C_n^k=\begin{cases}
\frac{n!}{k!(n-k)!},&\text{if } k\leq n,\\
0,& \text{if } k>n;
\end{cases}
$$ $r_{\mathbf A}$, $\textup{range}(\mathbf A)$, and  $\textup{ker}(\mathbf A)$ denote the rank, the range, and the null space of a matrix $\mathbf A$, respectively; $k_{\mathbf A}$ (the $k$-rank  of  $\mathbf A$ \cite[p. 162]{HarshmanLundy1984}) is the largest number  such that every subset of $k_{\mathbf A}$ columns of the matrix $\mathbf A$ is  linearly independent; 
``$\odot$'' and ``$\otimes$'' denote the Khatri-Rao  and  Kronecker product, respectively:
 \begin{align*}
 \mathbf A\odot\mathbf B &= [\mathbf a_1\otimes\mathbf b_1\ \dots\ \mathbf a_R\otimes\mathbf b_R ],\\
 \mathbf a\otimes\mathbf b &= [a_1b_1\dots a_1b_j\ \dots\ a_Ib_1\dots a_Ib_J]^T.
 \end{align*}
It is well known that PD \eqref{eqintro2} can be rewritten in a matrix form as 
  \begin{equation}\label{eq:(1.7)}
  \mathbf R_{1,0}(\mathcal T):=
  \left[
  \begin{matrix}
  \mathbf T_1\\
  \vdots\\
  \mathbf T_I
  \end{matrix}
  \right] =
 \left[
 \begin{matrix}
 \mathbf B\Diag{\mathbf a^1}\mathbf C^T\\
 \vdots\\
 \mathbf B\Diag{\mathbf a^I}\mathbf C^T
 \end{matrix}
 \right]= 
 (\mathbf A\odot\mathbf B)\mathbf C^T\in\mathbb F^{IJ\times K},
  \end{equation}
   where $\mathbf T_i:=(t_{ijk})_{j,k=1}^{J,K}$ denotes the $i$th horizontal slice of $\mathcal T=(t_{ijk})_{i,j,k=1}^{I,J,K}$, $\mathbf a^i:= [a_{i1}\ \dots\ a_{i_R}]$ denotes the $i$th row of $\mathbf A\in\mathbb F^{I\times R}$, and $\Diag{\mathbf a^i}$ denotes a square diagonal matrix with  the elements of the vector $\mathbf a^i$ on the main diagonal.

 To simplify the presentation and w.l.o.g. we will assume throughout 
 that the third dimension $K$  coincides with $r_{\mathbf C}$, yielding $r_{\mathbf C}=K\leq R$.
This can  always be achieved in a ``dimensionality reduction'' step (see, for instance, \cite[Subsection 1.4]{LinkGEVD}).
  \section{Previous results, new contribution, and organization of the paper} \label{subsubsection:mainconstruction}
To explain our contribution, we first briefly recall previous results on uniqueness conditions and algebraic algorithms.  (We  refer the readers to \cite{PartI,PartII,LinkGEVD} and  references therein for recent results and a detailed overview of early results.)
\subsection{At least two factor matrices have full column rank}  
 We say that a  matrix has full column rank if its columns are linearly independent, implying that it cannot have more columns than rows.
 The following result is well-known and goes back  to Kronecker and Weierstrass.
\begin{theorem}\cite{Harshman1972,Leurgans1993}\label{theorem:Harshman}
Let $\mathcal T=[\mathbf A,\mathbf B,\mathbf C]_R$ and suppose that the  matrices $\mathbf B$ and $\mathbf C$  have full column rank and that
 any two columns of $\mathbf A$ are linearly independent:
\begin{equation}\label{eq:Harshman_condition}
 r_{\mathbf B}=r_{\mathbf C}=R, \qquad \ k_{\mathbf A}\geq 2.
 \end{equation}
 Then $r_{\mathcal T}=R$, the CPD of $\mathcal T$ is unique and can be found algebraically.
 \end{theorem}
 Theorem \ref{theorem:Harshman} is the heart of the algebraic algorithms presented in \cite{LinkGEVD} and also in this paper.
To give an idea of how the CPD  in Theorem \ref{theorem:Harshman} is computed, let us consider the particular case of $2\times R\times R$ tensors.
Then, by \eqref{eq:Harshman_condition},   $\mathbf B$ and $\mathbf C$ are $R\times R$ nonsingular matrices. For simplicity we also assume that  the second row of $\mathbf A$ does not contain zero entries.
By \eqref{eq:(1.7)}, PD \eqref{eqintro2} can be rewritten as
\begin{equation}\label{eq:har1}
\mathbf T_1 = \mathbf B\Diag{\mathbf a^1}\mathbf C^T\ \text{ and }\ \mathbf T_2 = \mathbf B\Diag{\mathbf a^2}\mathbf C^T,
\end{equation}
 which easily implies that
 $$
 \mathbf T_1\mathbf T_2^{-1} = \mathbf B\mathbf D\mathbf B^{-1},\quad
 \mathbf T_1^T\mathbf T_2^{-T} = \mathbf C\mathbf D\mathbf C^{-1},
 $$
 where $\mathbf D = \Diag{\mathbf a^1}\Diag{\mathbf a^2}^{-1}$. By the assumption $k_{\mathbf A}\geq 2$,
 the diagonal entries of $\mathbf D$ are distinct. Hence, the matrices $\mathbf B$ and $\mathbf C$ can be uniquely identified
up to permutation and column scaling  from the eigenvalue decomposition of $\mathbf T_1\mathbf T_2^{-1}$ and $\mathbf T_1^T\mathbf T_2^{-T}$, respectively.
 One can then easily recover $\mathbf A$ from \eqref{eq:har1}.  
 Note that, in general, the matrices $\mathbf B$ and $\mathbf C$ in Theorem \ref{theorem:Harshman}  can be obtained from the GEVD of the matrix
  pencil $(\mathbf T_1,\mathbf T_2)$.
\subsection{At least one factor matrix has full column rank}\label{subsection2.2}
In this subsection we assume that only the third factor matrix of $\mathcal T$ has full column rank. 
It was shown in \cite{JiangSid2004} that PD \eqref{eqintro2} is unique if and only if
\begin{equation}\label{eq:U2}
r_{\mathbf A\Diag{\bm{\lambda}}\mathbf B^T}\geq 1 \ \text{for all }\bm{\lambda}=(\lambda_1,\dots,\lambda_R)\ \text{with at least two nonzero entries}.
\end{equation}
Condition \eqref{eq:U2} is not easy to check for a specific tensor.  The following condition is more restrictive but easy to check 
\cite{DeLathauwer2006,LinkGEVD}. We denote by $\mathcal C_m(\mathbf A)\in\mathbb R^{C^m_I\times C^m_R}$ 
the $m$th compound matrix of $\mathbf A\in\mathbb F^{I\times R}$, i.e.
the matrix containing the determinants of all $m\times m$ submatrices of $\mathbf A$ arranged with
the submatrix index sets in lexicographic order.
\begin{theorem}\cite{DeLathauwer2006,LinkGEVD}\label{th:C2}
Let $\mathcal T=[\mathbf A,\mathbf B,\mathbf C]_R$ and suppose that
\begin{equation}\label{eq:Lievenfcr}
 \text{the  matrices }\mathcal C_2(\mathbf A)\odot\mathcal C_2(\mathbf B)\ \text{ and } \mathbf C\ \text{  have full column rank.}
\end{equation}
 Then $r_{\mathcal T}=R$ and the CPD of $\mathcal T$ is unique.
\end{theorem}
It was shown in \cite{DeLathauwer2006,LinkGEVD} that the assumptions in Theorem \ref{th:C2} also imply an algebraic algorithm.
The algorithm is based on the following relation between $\mathcal T$ and its factor matrices:
\begin{equation}\label{eq:identityLieven}
\widetilde{\mathbf R}_{2,0}(\mathcal T)=(\mathcal C_2(\mathbf A)\odot\mathcal C_2(\mathbf B))\mathbf S_{2,0}(\mathbf C)^T,
\end{equation}
in which $\widetilde{\mathbf R}_{2,0}(\mathcal T)$ denotes an $C^2_IC^2_J\times R^2$ matrix whose
\begin{gather*}
 \left((j_1(2j_2-j_1-1)-2)I(I-1)/4+i_1(2i_2-i_1-1)/2, (r_2-1)R+r_1\right)\text{-th}\\
 (1\leq i_1<i_2\leq I,\ 1\leq j_1<j_2\leq J,\ 1\leq r_1,r_2\leq R)
\end{gather*}
entry is equal to 
$
t_{i_1j_1r_1}t_{i_2j_2r_2}+t_{i_1j_1r_2}t_{i_2j_2r_1}-t_{i_1j_2r_1}t_{i_2j_1r_2}-t_{i_1j_2r_2}t_{i_2j_1r_1}
$
and $\mathbf S_{2,0}(\mathbf C)$ denotes an $R^2\times C^2_R$ matrix that has columns 
$
\frac{1}{2}(\mathbf c_{r_1}\otimes\mathbf c_{r_2}+\mathbf c_{r_2}\otimes\mathbf c_{r_1})$, $1\leq r_1<r_2\leq R$.
Computationally, the identity \eqref{eq:identityLieven} is used as follows. First, the subspace $\ker(\widetilde{\mathbf R}_{2,0}(\mathcal T))$ is used to construct an auxiliary $R\times R\times R$ tensor $\mathcal W$ that has CPD $\mathcal W =[\mathbf C^{-T},\mathbf C^{-T},\mathbf M]_R$ in which both $\mathbf C^{-T}$ and $\mathbf M$ have full column rank. The CPD of $\mathcal W$ is computed as in Theorem \ref{theorem:Harshman},
which gives the matrix $\mathbf C^{-T}$.  The third factor  matrix of $\mathcal T$, $\mathbf C$, is obtained from  $\mathbf C^{-T}$ and
 the first two factor matrices $\mathbf A$ and $\mathbf B$ can  be easily found
  from $\mathbf R_{1,0}(\mathcal T)\mathbf C^{-T}=\mathbf A\odot\mathbf B$ (see \eqref{eq:(1.7)}) using the fact that the columns
 of $\mathbf A\odot\mathbf B$ are vectorized  rank-$1$ matrices.
 
\subsection{None of the factor matrices is required to have full column rank}\label{subsection2.3}
The following result is known as Kruskal's theorem. It is the most well-known result on uniqueness of the CPD.
\begin{theorem}\label{th:Kruskal}\cite{Kruskal1977}
Let $\mathcal T=[\mathbf A,\mathbf B,\mathbf C]_R$ and suppose that
\begin{equation}
2R+2\leq k_{\mathbf A}+k_{\mathbf B}+k_{\mathbf C}.
\label{eq:Kruskal}
\end{equation}
Then $r_{\mathcal T}=R$ and the CPD of $\mathcal T$ is unique.
\end{theorem}
 In \cite{PartI,PartII} we presented several generalizations of uniqueness Theorems \ref{th:C2} and \ref{th:Kruskal}.  In \cite{LinkGEVD}
we showed that the CPD can be computed algebraically under a much weaker assumption than  \eqref{eq:Kruskal}.
\begin{theorem}\label{th:LinkGEVD}\cite[Theorem 1.7]{LinkGEVD}
Let $\mathcal T=[\mathbf A,\mathbf B,\mathbf C]_R$ and suppose that
\begin{equation}
\mathcal C_m(\mathbf A)\odot\mathcal C_m(\mathbf B)\ \text{has full column rank for}\ m=R-k_{\mathbf C}+2.
\label{eq:Cncondition}
\end{equation}
Then $r_{\mathcal T}=R$, the CPD of $\mathcal T$ is unique and can be computed algebraically.
\end{theorem}
 The algorithm in \cite{LinkGEVD}
is based on the following extension of \eqref{eq:identityLieven}:
\begin{equation}\label{eq:LinkGEVDidentity}
\widetilde{\mathbf R}_{m,0}(\mathcal T)=(\mathcal C_m(\mathbf A)\odot\mathcal C_m(\mathbf B))\mathbf S_{m,0}(\mathbf C)^T,
\end{equation}
where the $C^m_IC^m_J\times K^m$ matrix $\widetilde{\mathbf R}_{m,0}(\mathcal T)$ is constructed from the given tensor
$\mathcal T$ and the $C^m_R\times K^m$ matrix $\mathbf S_{m,0}(\mathbf C)$ depends in a certain way on  $\mathbf C$. 
We refer the reader to \cite{LinkGEVD} for details on the algorithm. Here we just mention that 
assumption \eqref{eq:Cncondition} guarantees that the matrix $\mathbf C$ can be recovered from the subspace $\ker(\widetilde{\mathbf R}_{m,0}(\mathcal T))$.

\subsection{Generic uniqueness results from algebraic geometry}\label{subsection:2.4}
So far we have discussed deterministic conditions, which are expressed in terms of particular $\mathbf A,\mathbf B, \mathbf C$. On the other hand,
generic conditions are expressed in terms of dimensions and rank and hold ``with probability one''. Formally, 
we say that the CPD of a generic $I\times J\times K$ tensor of rank $R$ is unique 
if
$$
\mu\{(\mathbf A,\mathbf B,\mathbf C): \text{the CPD of the tensor}\ \mathcal T=[\mathbf A,\mathbf B,\mathbf C]_R\ \text{is not  unique}\}=0,
$$
where $\mu$ denotes the Lebesgue measure on $\mathbb F^{(I+J+K)R}$.

It is known from algebraic geometry  that if $2\leq I\leq J\leq K\leq R$, then each of the following
conditions implies that the CPD of a generic $I\times J\times K$ tensor of rank $R$ is unique:
\begin{align}
R&\leq\frac{I+J+2K-2-\sqrt{(I-J)^2+4K}}{2} & &(\text{see  \cite[Proposition 1.6]{AlgGeom1}}),  \label{eq:genericbound1}\\
R&\leq \frac{IJK}{I+J+K-2}-K,\ 3\leq I,\ \mathbb F=\mathbb C & & (\text{see \cite[Corollary 6.2]{Bocci2013}}), \label{eq:genericbound2}\\
R&\leq 2^{\alpha+\beta-2} \leq \frac{IJ}{4}& &(\text{see  \cite[Theorem 1.1]{ChiantiniandOttaviani}}),   \label{eq:genericbound3}
\end{align} 
where $\alpha$ and $\beta$ are maximal integers such that $2^\alpha\leq I$  and  $2^\beta\leq J$.
Bounds \eqref{eq:genericbound1}--\eqref{eq:genericbound3} complement each other. If $R=K$, then bound \eqref{eq:genericbound1} is equivalent to
\begin{equation}
R\leq(I-1)(J-1).\label{eq:U2generic}
\end{equation}
If $\mathbb F=\mathbb C$, then \eqref{eq:U2generic} is not only sufficient but also necessary, i.e., the decomposition is generically not unique  for $R >(I-1)(J-1)$ \cite[Proposition 2.2]{ChiantiniandOttaviani}. 

\subsection{Generic versions of deterministic uniqueness conditions}
Theorems \ref{th:C2}--\ref{th:LinkGEVD}, taken from \cite{DeLathauwer2006,LinkGEVD}, give deterministic  conditions under which
the CPD is unique and can be computed algebraically.  Generic counterparts of condition
\eqref{eq:Lievenfcr} and Kruskal's bound \eqref{eq:Kruskal}, for the case where\\ $\max(I,J,K)\leq R$,  are given by
\begin{align}
C^2_R &\leq C^2_I C^2_J\ \text{ and } R\leq K\qquad & &(\text{see \cite{DeLathauwer2006}})\text{ and}\label{eq:gen_counter1}\\
2R+2  &\leq I + J +K\qquad\qquad\qquad \qquad       & &(\text{trivial}),\label{eq:gen_counter2}
\end{align}
respectively.
We are not aware of a generic counterpart of condition \eqref{eq:Cncondition}, but, obviously,
\eqref{eq:Cncondition}  may hold only if the number of columns of
the matrix $\mathcal C_m(\mathbf A)\odot\mathcal C_m(\mathbf B)$ does not exceed the number of rows, i.e., if
\begin{equation}\label{eq:15}
C^m_R \leq C^m_I C^m_J, \text{ where } m=R-K+2.
\end{equation}
It can be verified that the algebraic geometry based bound \eqref{eq:genericbound1} significantly improves bounds \eqref{eq:gen_counter1}--\eqref{eq:15} if $\min(I,J)\geq 3$.
For instance, if $R=K$, then bound \eqref{eq:genericbound1} is equivalent to \eqref{eq:U2generic}, as has been mentioned earlier,
while \eqref{eq:gen_counter1} and \eqref{eq:gen_counter2} reduce to $R\leq(J-\frac{1}{2})(I-\frac{1}{2})/\sqrt{2}+1$ and $R\leq I+J-1$, respectively.
\subsection{Our contribution and organization of the paper}
In this paper we further extend results from \cite{DeLathauwer2006,PartI,PartII, LinkGEVD}, narrowing the gap with what is known from algebraic geometry. Namely, we present new deterministic conditions that
guarantee that the CPD is unique and can be computed algebraically. 
Although we do not formally prove that generically the  condition coincides  with \eqref{eq:genericbound1}, in our simulations we have been able to find the factor matrices  by algebraic means up to the latter bound (Examples  \ref{Example2.5} and
\ref{Example2.12}). Moreover, the algebraic scheme is shown to outperform numerical optimization
(Example \ref{example:manyinits}).

Key to our derivation is the following generalization of \eqref{eq:(1.7)}, \eqref{eq:identityLieven}, and \eqref{eq:LinkGEVDidentity}:
\begin{equation}
\mathbf R_{m,l}(\mathcal T) := {\mathbf\Phi}_{m,l}(\mathbf A,\mathbf B)\mathbf S_{m+l}(\mathbf C)^T,\ \ m\geq 1,\ \ l\geq 0,
\label{eq:mainidentity}
\end{equation}
in which  the matrices $\mathbf R_{m,l}(\mathcal T)$, $\mathbf{\Phi}_{m,l}(\mathbf A,\mathbf B)$, and
$\mathbf S_{m+l}(\mathbf C)$  are constructed from  the tensor $\mathcal T$, the matrices $\mathbf A$ and $\mathbf B$, and the matrix $\mathbf C$, respectively. 
The precise definitions of these matrices  are deferred to Section \ref{sec:constructions}, as they require additional technical notations.
In order to maintain the easy flow of the text presentation, the  proof of \eqref{eq:mainidentity} is  given in \ref{subsection:construction}.
The following scheme illustrates the links and shows that, to obtain our new results, we use \eqref{eq:mainidentity} for $m\geq 2$ and $l\geq 1$:

%
\begin{center}
\includegraphics[]{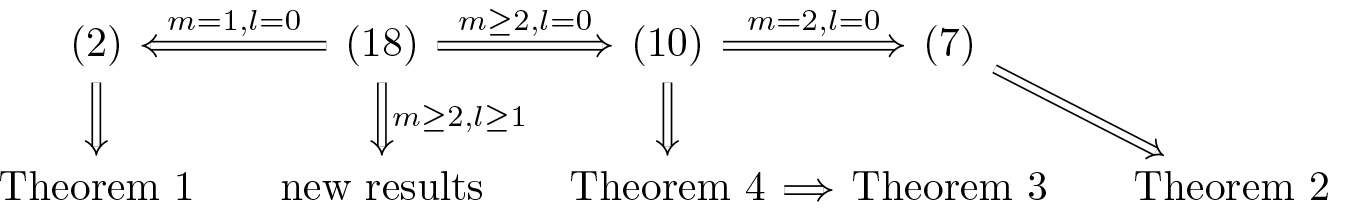}
\end{center}
(To clarify the link between $\eqref{eq:mainidentity}$ and $\eqref{eq:LinkGEVDidentity}$, we need to mention that the matrices $\widetilde{\mathbf R}_{m,0}(\mathcal T)$ and $\mathcal C_m(\mathbf A)\odot\mathcal C_m(\mathbf B)$ in \eqref{eq:LinkGEVDidentity} are obtained by removing the zero and redundant rows of
the matrices $\mathbf R_{m,0}(\mathcal T)$ and ${\mathbf\Phi}_{m,0}(\mathbf A,\mathbf B)$, respectively).

Our main results on uniqueness and algebraic algorithms for CPD are formulated, explained, and illustrated in Sections \ref{subsection:fcr}  (Theorem \ref{th:maintheoremfcr}) and \ref{subsectio:notfcr} (Theorems \ref{th:uniquenessonefm1}--\ref{th:uniquenessandalgorithmrefol}). Namely, in Sections \ref{subsection:fcr} and \ref{subsectio:notfcr} we generalize results mentioned in Subsections \ref{subsection2.2} and \ref{subsection2.3}, respectively. 
In particular, Theorem \ref{th:maintheoremfcr}
in Section \ref{subsection:fcr} is the special case of Theorem \ref{th:uniquenessandalgorithmrefol}
 in Section \ref{subsectio:notfcr}, where the third  factor matrix has full column rank, i.e.
 $r_{\mathbf C}=K=R$. For reasons of readability,  in our presentation we proceed from the easy
  Section \ref{subsection:fcr} ($r_{\mathbf C}=R$) to the more difficult Section \ref{subsectio:notfcr} ($r_{\mathbf C}\leq R$).
The proofs related to Sections \ref{subsection:fcr} and \ref{subsectio:notfcr} are given in Section \ref{proofsnfcr} and \ref{section:Appendix}.
In Section \ref{proofsnfcr} we  go from complicated to easy, i.e., in Subsections \ref{subsection6.1}--\ref{Subsection:last} we first prove the results related to Section \ref{subsectio:notfcr} and then  we derive Theorem \ref{th:maintheoremfcr}  from Theorem \ref{th:uniquenessandalgorithmrefol} in Subsection \ref{subsection6.6}.
The paper is concluded in Section \ref{sec:discussion}. 

Our presentation is in terms of real-valued tensors and real-valued factor matrices for notational convenience.
  Complex variants are easily obtained by taking into account complex conjugations.

\section{Construction of the matrices $\mathbf R_{m,l}(\mathcal T)$, $\mathbf{\Phi}_{m,l}(\mathbf A,\mathbf B)$, and
$\mathbf S_{m+l}(\mathbf C)$  }\label{sec:constructions}
Let us first introduce some additional notation.
 Throughout the paper $P_{\{l_1,\dots,l_k\}}$ denotes the set of all permutations of the set $\{l_1,\dots,l_k\}$.
 We follow the convention that if some of the values $l_1,\dots,l_k$ coincide, then
 the cardinality of $P_{\{l_1,\dots,l_k\}}$ is counted taken into account multiplicities, so that always $\textup{card}\ P_{\{l_1,\dots,l_k\}}=k!$.
 For instance, $P_{\{1,1,1\}}$ consists of  six identical entries $(1,1,1)$. One can easily check that any integer from $\{1, \dots,I^{m+l}J^{m+l}\}$ can be uniquely represented as
  $(\tilde i -1)J^{m+l}+\tilde j$ and that any integer from  $\{1,\dots,K^{m+l}\}$ can be uniquely represented
  as $\tilde k$, where 
   \begin{align}
      \tilde i &:= 1+\sum\limits_{p=1}^{m+l}(i_p-1)I^{m+l-p},  & & i_1,\dots,i_{m+l}\in\{1,\dots,I\},&\label{eq:tildei}\\
      \tilde j &:= 1+\sum\limits_{p=1}^{m+l}(j_p-1)J^{m+l-p},  & & j_1,\dots,j_{m+l}\in\{1,\dots,J\},&\label{eq:tildej}\\
      \tilde k &:= 1+\sum\limits_{p=1}^{m+l}(k_p-1)K^{m+l-p},  & & k_1,\dots,k_{m+l}\in\{1,\dots,K\}.&\label{eq:tildek}
    \end{align}
These expressions are useful for switching between  tensor, matrix and vector representations. We can now define   
$\mathbf R_{m,l}(\mathcal T)$ as follows.
 \begin{definition}\label{def:matrixtensor} Let $\mathcal T\in\mathbb R^{I\times J\times K}$. The $I^{m+l}J^{m+l}$-by-$K^{m+l}$ matrix whose
 $((\tilde i -1)J^{m+l}+\tilde j,\tilde k )$th entry is
  \begin{equation}
  \frac{1}{m!(m+l)!}
   \sum\limits_{\substack{(s_1,\dots,s_{m+l})\in\\ P_{\{k_1,\dots,k_{m+l}\}}}}
   \det\left[\begin{matrix}
   t_{i_1j_1s_1}&\dots & t_{i_1j_ms_m}\\
   \vdots       &\vdots& \vdots\\
   t_{i_mj_1s_1}&\dots & t_{i_mj_ms_m}
   \end{matrix}\right]\prod\limits_{p=1}^l
   t_{i_{m+p}j_{m+p}s_{m+p}}
   \label{eq:29}
   \end{equation}
   is denoted by  $\mathbf R_{m,l}(\mathcal T)$.
   \end{definition}
   The matrices  ${\mathbf\Phi}_{m,l}(\mathbf A,\mathbf B)$ and $\mathbf S_{m+l}(\mathbf C)$ will have
   $M(m,l,R)$ columns, where
     \begin{equation*}
      M(m,l,R) := C^m_RC^{m-1}_{m+l-1}+C^{m+1}_RC^{m}_{m+l-1}+\dots+C^{m+l}_RC^{m+l-1}_{m+l-1}. 
      \end{equation*} 
    The columns of these matrices are 
    indexed by $(m+l)$-tuples $(r_1,\dots,r_{m+l})$ such that
   \begin{equation}
   \begin{split}
     &\qquad 1\leq r_1\leq r_2\leq\dots\leq r_{m+l}\leq R \text{ and }\\
      \text{the set }&\{r_1,\dots,r_{m+l}\} \text{ contains at least } m \text{ distinct elements.} 
     \end{split}
     \label{eq:mplusltuples}
   \end{equation}
   It is easy to show that there indeed exist $M(m,l,R)$    
    $(m+l)$-tuples which  satisfy condition \eqref{eq:mplusltuples}.
    We follow the convention that the
   $(m+l)$-tuples in \eqref{eq:mplusltuples} are  ordered lexicographically:
     the  $(m+l)$-tuple $(r_1',\dots,r_{m+l}')$ is 
    preceding the  $(m+l)$-tuple $(r_1'',\dots,r_{m+l}'')$ if and only if either
    $r_1' < r_1''$ or there exists $k\in  \{1,\dots,m+l-1\}$ such that $r_1' = r_1'',\dots r_k' = r_k''$ and
    $r_{k+1}' < r_{k+1}''$.
    
    We can now define ${\mathbf\Phi}_{m,l}(\mathbf A,\mathbf B)$ and $\mathbf S_{m+l}(\mathbf C)$ as follows.
   \begin{definition}\label{def:matrixPhi}
   Let $\mathbf A\in\mathbb R^{I\times R}$, $\mathbf B\in\mathbb R^{J\times R}$. The $I^{m+l}J^{m+l}$-by-$M(m,l,R)$ matrix
   whose  $((\tilde i -1)J^{m+l}+\tilde j, (r_1,\dots,r_{m+l}))$th entry is
      \begin{equation*}
   \begin{split}
   \frac{1}{(m!)^2}
     \sum\limits_{\substack{(s_1,\dots,s_{m+l})\in\\ P_{\{r_1,\dots,r_{m+l}\}}}}
     \det&\left[\begin{matrix}
     a_{i_1s_1}&\dots & a_{i_1s_m}\\
     \vdots       &\vdots& \vdots\\
     a_{i_ms_1}&\dots & a_{i_ms_m}
     \end{matrix}\right]\cdot
     \det\left[\begin{matrix}
              b_{j_1s_1}&\dots & b_{j_1s_m}\\
              \vdots       &\vdots& \vdots\\
              b_{j_ms_1}&\dots & b_{j_ms_m}
              \end{matrix}\right]\cdot\\
    &\qquad  a_{i_{m+1}s_{m+1}}\cdots a_{i_{m+l}s_{m+l}}\cdot
    b_{j_{m+1}s_{m+1}}\cdots b_{j_{m+l}s_{m+l}}
     \end{split} 
     \end{equation*}
        is denoted by $\mathbf \Phi_{m,l}(\mathbf A,\mathbf B)$.
   \end{definition}
  \begin{definition}\label{def:matrixS}
      Let $\mathbf C\in\mathbb R^{K\times R}$. The $K^{m+l}$-by-$M(m,l,R)$ matrix
      whose\\  $(r_1,\dots,r_{m+l})$th column is
      \begin{equation}
      \frac{1}{(m+l)!} \sum\limits_{(s_1,\dots,s_{m+l})\in\ P_{\{r_1,\dots,r_{m+l}\}}}\mathbf c_{s_1}\otimes\dots\otimes\mathbf c_{s_{m+l}}
      \label{eq:def2.3star}
     \end{equation}
      is denoted by $\mathbf S_{m+l}(\mathbf C)$.
     \end{definition}
 \section{At least one factor matrix of $\mathcal T$ has full column rank}\label{subsection:fcr}
In this section we  generalize results from Subsection \ref{subsection2.2}, i.e. we assume that
 the matrix $\mathbf C$ has full column rank and without loss of generality $r_{\mathbf C}=K=R$. The more general case
 $r_{\mathbf C}=K\leq R$  is handled in Theorem \ref{th:uniquenessandalgorithmrefol} in Section \ref{subsectio:notfcr}.
 The goal of this section is to explain why and how the algebraic algorithm works in the
relatively easy but important
  case $r_{\mathbf C}=K=R$,  so that in turn
 Section \ref{subsectio:notfcr} will be more accessible.
 
It can be shown that for $l=0$,  
 condition \eqref{eq:maincondfcr} in Theorem \ref{th:maintheoremfcr} below reduces to condition \eqref{eq:Lievenfcr}. Thus,
 Theorem \ref{th:C2} is the special case of Theorem \ref{th:maintheoremfcr}  corresponding to $l=0$.
 The simulations in Example \ref{Example2.5} below indicate that  it is always possible to find some $l\geq 0$
 so that  \eqref{eq:maincondfcr}  also covers \eqref{eq:U2}.
 Although there is no general  proof, this suggests that  \eqref{eq:U2} can always be verified by checking \eqref{eq:maincondfcr} for some $l\geq 0$.
 This would  imply that  Algorithm \ref{alg:fcr} can compute the CPD of a generic tensor up to the necessary condition $R\leq(I-1)(J-1)$.
 Example \ref{Example2.5} confirms this up to $R\leq 24$.
  
Let $S^{m+l}(\mathbb R^{K^{m+l}})\subset \mathbb R^{K^{m+l}}$ denote the subspace spanned by all vectors of the form $\mathbf x\otimes\dots\otimes\mathbf x$, where $\mathbf x\in\mathbb R^K$ is repeated $m+l$ times. In other words, $S^{m+l}(\mathbb R^{K^{m+l}})$ contains  vectorized versions of all  $K\times\dots\times K$  symmetric tensors of order $m+l$, yielding $\dim S^{m+l}(\mathbb R^{K^{m+l}}) =C^{m+l}_{K+m+l-1}$. We  have the  following result.
 \begin{theorem}\label{th:maintheoremfcr}
 Let $\mathcal T=(t_{ijk})_{i,j,k=1}^{I,J,K}=[\mathbf A,\mathbf B,\mathbf C]_R$, $r_{\mathbf C}=K=R$, $l\geq 0$, and let
 the matrix $\mathbf R_{2,l}(\mathcal T)$ be defined as in Definition \ref{def:matrixtensor}.    
 Assume that
\begin{align}
 \dim  \left(\ker (\mathbf R_{2,l}(\mathcal T))\bigcap S^{2+l}(\mathbb R^{K^{2+l}})\right)=R.\label{eq:maincondfcr}
\end{align}
 Then 
 \begin{itemize}
 \item[(1)] $r_{\mathcal T} = R$ and  the  CPD of $\mathcal T$ is unique; and
 \item[(2)] the CPD of $\mathcal T$ can be found algebraically.
 \end{itemize}
  \end{theorem}
Condition \eqref{eq:maincondfcr} in Theorem \ref{th:maintheoremfcr} means that the intersection  
of $\ker (\mathbf R_{2,l}(\mathcal T))$ and $S^{2+l}(\mathbb R^{K^{2+l}})$
has the minimal possible dimension. Indeed, by  \eqref{eq:mainidentity}, Definition \ref{def:matrixS}, and the assumption $r_{\mathbf C}=K=R$, 
we have that the intersection contains at least $R$ linearly independent vectors:
      \begin{equation*}
            \begin{split}
      &\ker (\mathbf R_{2,l}(\mathcal T))\bigcap S^{2+l}(\mathbb R^{K^{2+l}})=
      \ker ({\mathbf\Phi}_{2,l}(\mathbf A,\mathbf B)\mathbf S_{2+l}(\mathbf C)^T)\bigcap S^{2+l}(\mathbb R^{K^{2+l}})\supseteq \\
      &\ker (\mathbf S_{2+l}(\mathbf C)^T)\bigcap S^{2+l}(\mathbb R^{K^{2+l}})\ni \mathbf x\otimes\dots\otimes\mathbf x,\quad \mathbf x\ \text{ is a column of }\ \mathbf C^{-T}.
       \end{split}
       \end{equation*}
       
 The  procedure that constitutes the proof of Theorem \ref{th:maintheoremfcr}(2) is summarized
 as Algorithm \ref{alg:fcr}. Let us comment on the different steps.
 From Definition \ref{def:matrixtensor} it follows that the rows of the matrix 
 $\mathbf R_{2,l}(\mathcal T)$ are vectorized versions of   $K\times\dots\times K$  symmetric tensors of order $2+l$. Consistently,  in step 2,
  we find the vectors $\mathbf w_1,\dots,\mathbf w_R$  that form a basis  of the orthogonal complement to $\textup{range}( \mathbf R_{2,l}(\mathcal T)^T)$ in the space $S^{2+l}(\mathbb R^{R^{2+l}})$. 
  In other words, $\operatorname{span}\{\mathbf w_1,\dots,\mathbf w_R\} = \ker (\mathbf R_{2,l}(\mathcal T))\bigcap S^{2+l}(\mathbb R^{K^{2+l}})$.
  If this subspace has  minimal dimension, then its structure provides a key to the estimation of $\mathbf C$. Indeed, we have already
  explained  that the minimal subspace is given by 
       \begin{equation}
     \label{eq:CCC}
     \ker( \mathbf R_{2,l}(\mathcal T))\bigcap S^{2+l}(\mathbb R^{R^{2+l}}) = \textup{range}\left(\underbrace{\mathbf C^{-T}\odot\dots\odot\mathbf C^{-T}}_{2+l}\right).
     \end{equation}
  In steps 4--5 we recover $\mathbf C^{-T}$ from $\mathbf W$ using \eqref{eq:CCC} as follows. By \eqref{eq:CCC},
   there exists a unique nonsingular $R\times R$ matrix $\mathbf M$ such that
 \begin{equation}
     \label{eq:WeqCCCM}
     \mathbf W = \left(\mathbf C^{-T}\odot\dots\odot\mathbf C^{-T}\right)\mathbf M^T.
      \end{equation}
  In step 4, we construct the tensor $\mathcal W$ whose  vectorized frontal slices are the vectors $\mathbf w_1,\dots,\mathbf w_R$. 
  Reshaping both sides of \eqref{eq:WeqCCCM}  we obtain the CPD $\mathcal W=[\mathbf C^{-T}, \mathbf C^{-T}\odot\dots\odot\mathbf C^{-T},\mathbf M]_R$.
  In step 5, we find the CPD by means of a GEVD   using the fact that all factor matrices of  $\mathcal W$ have full column rank, i.e., we have reduced the problem to a situation that is covered by the basic  Theorem \ref{theorem:Harshman}. Finally, in step 6 we recover $\mathbf A$ and $\mathbf B$
  from $\mathbf R_{1,0}(\mathcal T)\mathbf C^{-T}=\mathbf A\odot\mathbf B$  using the fact that the columns
   of $\mathbf A\odot\mathbf B$ are vectorized  rank-$1$ matrices.
  \begin{algorithm}
  \caption{(Computation of CPD, $K=R$ (see Theorem \ref{th:maintheoremfcr}(ii)))}
  \label{alg:fcr}
  \begin{algorithmic}[1]
  \Require $\mathcal T\in\mathbb R^{I\times J\times R}$ and $l\geq 0$  with the property that  there exist $\mathbf A\in\mathbb R^{I\times R}$, $\mathbf B\in\mathbb R^{J\times R}$, and $\mathbf C\in\mathbb R^{R\times R}$ such that $R\geq 2$, $\mathcal T=[\mathbf A,\mathbf B,\mathbf C]_R$, $r_{\mathbf C}=R$, and \eqref{eq:maincondfcr} holds.
  \Ensure Matrices $\mathbf A\in\mathbb R^{I\times R}$, $\mathbf B\in\mathbb R^{J\times R}$ and $\mathbf C\in\mathbb R^{R\times R}$ such that $\mathcal T=[\mathbf A,\mathbf B,\mathbf C]_R$
   \State Construct  the $I^{2+l}J^{2+l}\times R^{2+l}$ matrix
  $\mathbf R_{2,l}(\mathcal T)$ by Definition \ref{def:matrixtensor}. 
   \State Find $\mathbf w_1,\dots,\mathbf w_R$ that form a basis  of $\ker( \mathbf R_{2,l}(\mathcal T))\bigcap S^{2+l}(\mathbb R^{R^{2+l}})$
  \State $\mathbf W\leftarrow [\mathbf w_1\ \dots\ \mathbf w_R]$
  \State Reshape the $R^{2+l}\times R$ matrix $\mathbf W$ into an $R\times R^{1+l}\times R$ tensor $\mathcal W$
  \State Compute the CPD
  \Statex $\mathcal W=[\mathbf C^{-T}, \mathbf C^{-T}\odot\dots\odot\mathbf C^{-T},\mathbf M]_R$\quad  ({\em $\mathbf M$ is a by-product})\ \ \ \ \ \ \ \ \ \ (GEVD)
   \State Find the columns of $\mathbf A$ and $\mathbf B$ from the equation $\mathbf A\odot\mathbf B=\mathbf R_{1,0}(\mathcal T)\mathbf C^{-T}$
  \end{algorithmic}
  \end{algorithm}
  
The following example demonstrates that the CPD  can  effectively be computed by Algorithm \ref{alg:fcr} for
$R\leq \min((I-1)(J-1),24)$.

\begin{example}\label{Example2.5}
We consider $I\times J\times (I-1)(J-1)$ tensors generated as a sum of  
$R=(I-1)(J-1)$ random  rank-$1$ tensors. More precisely, the tensors  are generated by a PD
$[\mathbf A,\mathbf B,\mathbf C]_R$ in which the  entries of
$\mathbf A$, $\mathbf B$, and $\mathbf C$ are independently drawn from the standard normal distribution $N(0,1)$.
We try different values $l=0,1,\dots,$ until condition \eqref{eq:maincondfcr} is met (assuming that this will be the case for some $l\geq 0$).
We test all cases $I\times J\times (I-1)(J-1)$ such that $I\geq 3$, $J\geq 3$, and 
$(I-1)(J-1)\leq 24$. The results are shown in Table \ref{table2}.
In all cases  \eqref{eq:maincondfcr} indeed holds for some $l\leq 2$; the actual value of
$l$ does not depend on the random trial, i.e., it  is constant for tensors of the same dimensions and rank. By comparison, the algebraic algorithm from \cite{DeLathauwer2006,LinkGEVD} is limited to the cases where $l=0$, which is not always sufficient to reach the bound $R\leq(I-1)(J-1)$.
In our implementation,  we retrieved  the vectors $\mathbf w_1,\dots,\mathbf w_R$ from the $R$-dimensional null space of  a
$C_{R+l+1}^{2+l}\times C_{R+l+1}^{2+l}$ positive  semi-definite  matrix $\mathbf Q$. The storage of $\mathbf Q$ is
the main bottleneck in our implementation.
To give some insight in the  complexity of the algorithm we included the computational time (averaged over 100 random tensors)
and the size of $\mathbf Q$ in the table. We implemented Algorithm \ref{alg:fcr} in MATLAB  2014a (the implementation was not optimized), and we did experiments on a computer with Intel Core 2 Quad CPUQ9650\@ 3.00 GHz$\times$4 and 8GB memory running Ubuntu 12.04.5 LTS. 
 \end{example}
 \begin{table}[!h]
       \begin{center}
    \caption{ Values of parameter $l$ in Theorem \ref{th:maintheoremfcr} and computational cost of Algorithm \ref{alg:fcr} for
    $I\times J\times (I-1)(J-1)$ tensors of rank $R=(I-1)(J-1)\leq 24$ (see Example \ref{Example2.5} for details). Note that 
    the CPD is not generically unique if $R>(I-1)(J-1)$ (see Subsection \ref{subsection:2.4}). In all cases a value of $l$ is found such that
    Algorithm \ref{alg:fcr} can be used. The rows with $l\geq 1$ are new results.
    }
    \begin{tabular}{|c|c|c|c|}
    \hline
    dimensions of $\mathcal T$&    $l$&  $C_{R+l+1}^{2+l}$ &  computational time (sec)\\
    \hline
    $3\times 3\times 4$       &0& 10 & 0.02\\
    $3\times 4\times 6$       &0& 21 & 0.035\\
    $3\times 5\times 8$       &0& 36 & 0.051\\
    $3\times 6\times 10$     &0& 55 & 0.074\\
    $3 \times 7 \times 12$   &1& 364 & 0.403\\         
    $3 \times 8 \times 14$   &1& 560    &   0.796\\      
    $3 \times 9 \times 16$   &1& 816   &    1.498\\
    $3 \times 10 \times 18$   &1& 1140   &   2.617\\     
    $3 \times 11 \times 20$   &1& 1540   &    5.032\\     
    $3 \times 12 \times 22$   &1& 2024   &   7.089\\  
    $3 \times 13 \times 24$   &1& 2600   &     11.084\\   
    $4\times 4\times 9$       &0& 45 &  0.06\\
    $4 \times 5 \times 12$   &1& 364     &   0.401\\ 
    $4 \times 6 \times 15$   &1& 680    &  1.096\\
    $4 \times 7 \times 18$   &2&  5985   &    30.941\\          
    $4 \times 8 \times 21$   &2&  10626   &     93.03\\
    $4 \times 9 \times 24$   &2&  17550   & 360.279 \\
    $5 \times 5 \times 16$   &1& 816   & 1.473 \\      
    $5 \times 6 \times 20$   &2&  8855   & 64.116 \\               
    $5 \times 7 \times 24$   &2&  17550   & 351.968 \\
    \hline
    \end{tabular}
    \label{table2}
    \end{center}
    \end{table}
The next example illustrates that  Algorithm  \ref{alg:fcr} may outperform optimization algorithms.
 \begin{example}\label{example:manyinits}
  Let $\mathcal T=[\mathbf A,\mathbf B, \mathbf C]_{12}\in\mathbb R^{3\times 7\times 12}$, with
  \begin{equation*}
  \begin{split}
  \mathbf A &= hankel((1,2,3),(3, 5, 7, 0, 6, 6, 7, 9, 0, 8, 2, 1)^T),\\
   \mathbf B &= [\mathbf I_7\ hankel((1,2,3,4,5,6,7), (7,0,1,2,3)^T)],\qquad
   \mathbf C = \mathbf I_{12},
   \end{split}
   \end{equation*}
   where $hankel(\mathbf c,\mathbf r^T)$ denotes a Hankel matrix whose first column is $\mathbf c$ and whose last row is $\mathbf r^T$. It turns out that \eqref{eq:maincondfcr} holds for $l=1$. It takes  less than $1$ second to compute the CPD of $\mathcal T$ by Algorithm \ref{alg:fcr}.
   On the other hand, it proves to be very difficult to find the CPD by means of numerical optimization.
   Among other optimization-based algorithms we tested the Gauss-Newton  dogleg trust region method  \cite{Sorber}.
   The algorithm was restarted  $500$ times from various random initial positions.
   %
   %
   In only 4 cases  the residual\\ $\|\mathcal T-[\mathbf A_{est},\mathbf B_{est},\mathbf C_{est}]_{12}\|/\|\mathcal T\|$ after $10000$ iterations 
   was   of the order of $0.0001$ and in all cases the estimated factor matrices were far from the true  matrices.
Other optimization-based algorithms did not yield better results.
    \end{example}
  \section{None of the factor matrices  is required to have full column rank}\label{subsectio:notfcr}
  In this subsection we consider the PD  $\mathcal T=(t_{ijk})_{i,j,k=1}^{I,J,K}=[\mathbf A,\mathbf B,\mathbf C]_R$ and extend results of the previous subsection to the case $r_{\mathbf C}=K\leq R$.
 \subsection{Results on uniqueness of one factor matrix and overall CPD}
We have two results on uniqueness of the third factor matrix.
 \begin{theorem}\label{th:uniquenessonefm1}
  Let $\mathcal T=(t_{ijk})_{i,j,k=1}^{I,J,K}=[\mathbf A,\mathbf B,\mathbf C]_R$, $r_{\mathbf C}=K\leq R$, $m=R-K+2$, and $l_1,\dots,l_m$
  be nonnegative integers. Let also
the matrices  $\mathbf \Phi_{1,l_1}(\mathbf A,\mathbf B),\dots,\mathbf \Phi_{m,l_m}(\mathbf A,\mathbf B)$ and
$\mathbf S_{1+l_1}(\mathbf C),\ \dots,\ \mathbf S_{m+l_m}(\mathbf C)$ be defined as in
Definition \ref{def:matrixPhi} and Definition \ref{def:matrixS}, respectively. Let 
$\mathbf U_1,\dots,\mathbf U_m$  be matrices such that their columns  form bases for
  $\textup{range}(\mathbf S_{1+l_1}(\mathbf C)^T), \dots, \textup{range}(\mathbf S_{m+l_m}(\mathbf C)^T)$,  respectively.
   Assume that
 \begin{itemize}
 \item[\textup{(i)}] $k_{\mathbf C}\geq 1$; and
 \item[\textup{(ii)}] $\mathbf A\odot\mathbf B$ has full column rank; and
 \item[\textup{(iii)}]  $\mathbf \Phi_{1,l_1}(\mathbf A,\mathbf B)\mathbf U_1$, \dots, $\mathbf \Phi_{m,l_m}(\mathbf A,\mathbf B)\mathbf U_m$ have full column rank.
 \end{itemize}
 Then $r_{\mathcal T} = R$ and  the  third factor matrix of $\mathcal T$ is unique. 
 \end{theorem}
 
  According to the following theorem the set of matrices in \textup{(iii)} in Theorem \ref{th:uniquenessonefm1} can be reduced to a single matrix
  if  $R\leq \min(k_{\mathbf A}, k_{\mathbf B})+K-1$.
  
      \begin{theorem}\label{th:uniquenessonefm2}
   Let $\mathcal T=(t_{ijk})_{i,j,k=1}^{I,J,K}=[\mathbf A,\mathbf B,\mathbf C]_R$, $r_{\mathbf C}=K\leq R$, $m=R-K+2$, and $l\geq 0$.
   Let also  the matrices $\mathbf \Phi_{m,l}(\mathbf A,\mathbf B)$ and 
   $\mathbf S_{m+l}(\mathbf C)$
   be defined as in Definition \ref{def:matrixPhi} and Definition \ref{def:matrixS}, respectively.
   Let $\mathbf U_m$ be a matrix such that its columns  form a basis for
   $\textup{range}(\mathbf S_{m+l}(\mathbf C)^T)$.
  Assume that
  \begin{itemize}
  \item[\textup{(i)}] $k_{\mathbf C}\geq 1$; and
  \item[\textup{(ii)}] $\mathbf A\odot\mathbf B$ has full column rank; and
  \item[\textup{(iii)}] $\min(k_{\mathbf A}, k_{\mathbf B})\geq m-1$; and
  \item[\textup{(iv)}]  the matrix $\mathbf \Phi_{m,l}(\mathbf A,\mathbf B)\mathbf U_m$ has full column rank.
  \end{itemize}
  Then $r_{\mathcal T} = R$ and  the  third factor matrix of $\mathcal T$ is unique. 
  \end{theorem} 
  
   The assumptions in Theorems \ref{th:uniquenessonefm1} and \ref{th:uniquenessonefm2} complement each other as follows: 
   in Theorem \ref{th:uniquenessonefm1} we do not require that the condition $\min(k_{\mathbf A}, k_{\mathbf B})\geq m-1$ holds while in Theorem \ref{th:uniquenessonefm2} we do not require that
   the matrices $\mathbf \Phi_{k,l_k}(\mathbf A,\mathbf B)\mathbf U_k$, $1\leq k\leq m-1$
   have full column rank. 
   
    It was shown in \cite[Proposition 1.20]{PartII} that  if $\mathcal T $
      has two PDs $\mathcal T=$$[\mathbf A,\mathbf B,\mathbf C]_R$ and $\mathcal T=[\bar{\mathbf A},\bar{\mathbf B},\mathbf C]_R$ that share the factor matrix $\mathbf C$ and if
      the condition
      \begin{equation}
      \max(\min(k_{\mathbf A},k_{\mathbf B}-1),\ \min(k_{\mathbf A}-1,k_{\mathbf B}))+k_{\mathbf C}\geq R+1\label{eq:uniqviaonefm}
      \end{equation}   
      holds, then both PDs consist of the same rank-one terms. 
      Thus, combining Theorems \ref{th:uniquenessonefm1}--\ref{th:uniquenessonefm2}
      with \cite[Proposition 1.20]{PartII} we directly obtain the following result on  uniqueness of the overall CPD.
      \begin{theorem}\label{theorem2.9}
      Let the assumptions in Theorem  \ref{th:uniquenessonefm1} or Theorem \ref{th:uniquenessonefm2}
      hold and let condition \eqref{eq:uniqviaonefm} be satisfied. Then $r_{\mathcal T} = R$ and 
      the CPD of tensor $\mathcal T$ is unique.
      \end{theorem}
 \subsection{Algebraic algorithm for CPD}\label{subsection2.4.2}
 We have the following result on algebraic computation.
  \begin{theorem}\label{th:uniquenessandalgorithm}
    Let $\mathcal T=(t_{ijk})_{i,j,k=1}^{I,J,K}=[\mathbf A,\mathbf B,\mathbf C]_R$, $r_{\mathbf C}=K\leq R$, $m=R-K+2$, and $l\geq 0$.
    Let also  the matrices $\mathbf \Phi_{m,l}(\mathbf A,\mathbf B)$ and 
    $\mathbf S_{m+l}(\mathbf C)$
    be defined as in Definition \ref{def:matrixPhi} and Definition \ref{def:matrixS}, respectively.
    Let $\mathbf U_m$ be a matrix such that its columns  form a basis for
    $\textup{range}(\mathbf S_{m+l}(\mathbf C)^T)$.
   Assume that
   \begin{itemize}
   \item[\textup{(i)}] $k_{\mathbf C}=K$; and
   \item[\textup{(ii)}] $\mathbf A\odot\mathbf B$ has full column rank; and
   \item[\textup{(iii)}]  the matrix $\mathbf \Phi_{m,l}(\mathbf A,\mathbf B)\mathbf U_m$ has full column rank.
   \end{itemize}
   Then $r_{\mathcal T} = R$, the CPD  of $\mathcal T$ is unique and can be found algebraically. 
   \end{theorem}
   
 The assumptions in Theorem \ref{th:uniquenessandalgorithm} are more restrictive than the assumptions
 in Theorem \ref{theorem2.9} as will be clear from
 Section \ref{proofsnfcr}. Hence,  the statement on rank and uniqueness in Theorem \ref{th:uniquenessandalgorithm}
 follows from Theorem \ref{theorem2.9}.
 To prove the statement on algebraic computation we will explain in Section \ref{proofsnfcr} that Theorem \ref{th:uniquenessandalgorithm} can be reformulated as follows (see Section \ref{subsection:fcr} for the definition of $S^{m+l}(\mathbb R^{K^{m+l}})$).
 \begin{theorem}\label{th:uniquenessandalgorithmrefol}
     Let $\mathcal T=(t_{ijk})_{i,j,k=1}^{I,J,K}=[\mathbf A,\mathbf B,\mathbf C]_R$, $r_{\mathbf C}=K\leq R$, $m=R-K+2$, and $l\geq 0$.
     Let also  the matrix $\mathbf R_{m,l}(\mathcal T)$ be defined as in Definition \ref{def:matrixtensor}.    
     Assume that
    \begin{itemize}
    \item[\textup{(i)}] $k_{\mathbf C}=K$; and
    \item[\textup{(ii)}] $\mathbf A\odot\mathbf B$ has full column rank; and
    \item[\textup{(iii)}]  $\dim  \left(\ker (\mathbf R_{m,l}(\mathcal T))\bigcap S^{m+l}(\mathbb R^{K^{m+l}})\right)=C^{K-1}_R$.
    \end{itemize}
    Then $r_{\mathcal T} = R$, the CPD  of $\mathcal T$ is unique and can be found algebraically. 
    \end{theorem}

    Note that if $k_{\mathbf C}=K$, then  by  \eqref{eq:mainidentity} and Lemma \ref{lemma:suppl} \textup{(i)} below,
    \begin{equation}\label{equation:2.19}
          \begin{split}
    &\dim  \left(\ker (\mathbf R_{m,l}(\mathcal T))\bigcap S^{m+l}(\mathbb R^{K^{m+l}})\right)=\\
    &\dim  \left(\ker ({\mathbf\Phi}_{m,l}(\mathbf A,\mathbf B)\mathbf S_{m+l}(\mathbf C)^T)\bigcap S^{m+l}(\mathbb R^{K^{m+l}})\right)\geq \\
    &\dim  \left(\ker (\mathbf S_{m+l}(\mathbf C)^T)\bigcap S^{m+l}(\mathbb R^{K^{m+l}})\right)=C^{K-1}_R.
     \end{split}
     \end{equation}
    Thus, assumption \textup{(iii)} of Theorem \ref{th:uniquenessandalgorithmrefol} means that we require the subspace to have the minimal possible dimension. That is, we suppose that the factor matrices $\mathbf A$, $\mathbf B$, and $\mathbf C$ are such that the multiplication by ${\mathbf\Phi}_{m,l}(\mathbf A,\mathbf B)$ in \eqref{eq:mainidentity} does not increase the overlap  between
     $\ker (\mathbf S_{m+l}(\mathbf C)^T)$ and $S^{m+l}(\mathbb R^{K^{m+l}})$. In other words, we suppose that the multiplication by
     ${\mathbf\Phi}_{m,l}(\mathbf A,\mathbf B)$ does not cause additional vectorized $K\times\dots\times K$ symmetric tensors of order $m+l$
     to be part of the null space of $\mathbf R_{m,l}(\mathcal T)$. This is key to the derivation. By the assumption, as we will explain further in this section, the only vectorized symmetric tensors in the null space of $\mathbf R_{m,l}(\mathcal T)$ admit a direct connection with the factor matrix $\mathbf C$, from which $\mathbf C$ may be retrieved. On the other hand, the null space of $\mathbf R_{m,l}(\mathcal T)$
     can obviously be computed from the given tensor $\mathcal T$.

      The algebraic procedure based on Theorem \ref{th:uniquenessandalgorithmrefol} consists of three phases and is summarized in  Algorithm \ref{alg:onefactormatrix}.      
     In  the first phase  we find  the $K\times C^{K-1}_R$ matrix
    $\mathbf F$ such that
    \begin{gather}
     \textup{every column of}\ \mathbf F \ \text{is orthogonal to exactly} \ K-1\ \text{columns of }\ \mathbf C\ \text{and}
     \label{P1}\\
     \begin{split}
     \text{any vector that is orthogonal to exactly } K-1 \text{ columns of } \mathbf C \\
         \qquad\qquad\qquad\text{ is proportional to a column of } \mathbf F.
               \end{split}\label{P2}
        \end{gather}
        \begin{algorithm}[!h]
        \caption{(Computation of CPD, $K\leq R$ (see Theorem \ref{th:uniquenessandalgorithmrefol}))}
        \label{alg:onefactormatrix}
        \begin{algorithmic}[1]
        \Require
        $\mathcal T\in\mathbb R^{I\times J\times K}$ and $l\geq 0$  with the property that  there exist $\mathbf A\in\mathbb R^{I\times R}$, $\mathbf B\in\mathbb R^{J\times R}$, and $\mathbf C\in\mathbb R^{K\times R}$ such that  $\mathcal T=[\mathbf A,\mathbf B,\mathbf C]_R$ and assumptions \textup{(i)}--\textup{(iii)} in Theorem \ref{th:uniquenessandalgorithmrefol} hold.
        \Ensure Matrices $\mathbf A\in\mathbb R^{I\times R}$, $\mathbf B\in\mathbb R^{J\times R}$ and $\mathbf C\in\mathbb R^{R\times R}$ such that $\mathcal T=[\mathbf A,\mathbf B,\mathbf C]_R$
        \Algphase{{\bf Phase 1:} Find the matrix $\mathbf F\in\mathbb R^{K\times C^{K-1}_R}$ such that $\mathbf F$ coincides with $\mathcal B(\mathbf C)$ up to (unknown) column permutation and  scaling}
        \State  Construct  the $I^{m+l}J^{m+l}\times K^{m+l}$ matrix
         $\mathbf R_{m,l}(\mathcal T)$ by Definition \ref{def:matrixtensor}. 
        \State Find $\mathbf w_1,\dots,\mathbf w_{C^{K-1}_R}$ that form a basis  of $\ker (\mathbf R_{m,l}(\mathcal T))\bigcap S^{m+l}(\mathbb R^{K^{m+l}})$
        \State  $\mathbf W\leftarrow [\mathbf w_1\ \dots\ \mathbf w_{C^{K-1}_R}]$
        \State Reshape the $K^{m+l}\times C^{K-1}_R$ matrix $\mathbf W$ into an $K\times K^{m+l-1}\times C^{K-1}_R$ tensor $\mathcal W$
         \State Compute the CPD
         \Statex \qquad $\mathcal W=[\mathbf F, \mathbf F\odot\dots\odot\mathbf F,\mathbf M]_{C^{K-1}_R}$\quad  ({\em $\mathbf M$ is a by-product})\ \ \ \ \ \ \ \ \ \ \qquad (GEVD)
        \Algphase{{\bf Phase 2 and Phase 3  (can be taken verbatim from \cite[Algorithms 1,2]{LinkGEVD})}}
        \end{algorithmic}
        \end{algorithm}
       Since $k_{\mathbf C}=K$ any $K-1$ columns  of $\mathbf C$ define a unique  column of $\mathbf F$ (up to scaling). Thus, \eqref{P1}--\eqref{P2} define the matrix  $\mathbf F$ up to column permutation and scaling. 
       A special representation of $\mathbf F$ (called $\mathcal B(\mathbf C)$) was studied in \cite{LinkGEVD}. It was shown in \cite{LinkGEVD} that  the matrix $\mathbf F$ can be considered as an unconventional variant of the inverse of $\mathbf C$:
    \begin{align}
     &\text{ every column of } \mathbf C \text{ is orthogonal to exactly } C^{K-2}_{R-1} \text{ columns of } \mathbf F,\label{P3}\\
    &\begin{multlined}
    \text{any vector that is orthogonal to exactly } C^{K-2}_{R-1} \text{ columns of } \mathbf F\\
    \text{ is proportional to a column of } \mathbf C.
     \end{multlined}\label{P4}
    \end{align}
    (Note that, since $k_{\mathbf C}=K$, multiplication by the Moore--Penrose pseudo-inverse $\mathbf C^{\dagger}$ yields $\mathbf C\mathbf C^{\dagger}=\mathbf I_K$. In contrast, for $\mathbf F$ we consider the product $\mathbf F\mathbf C$.)
     It can be shown (see Lemma \ref{lemma4.5}) that under the assumptions in Theorems \ref{th:uniquenessandalgorithm}--\ref{th:uniquenessandalgorithmrefol}:
        \begin{gather}
        k_{\mathbf F}\geq 2,\qquad \text{the matrix }\mathbf F^{(m+l-1)}\ \text{has full column rank and}\ \label{eq:beforeFFF}
        \\
        \label{eq:FFF}
        \ker( \mathbf R_{m,l}(\mathcal T))\bigcap S^{m+l}(\mathbb R^{K^{m+l}}) = \textup{range}\left(\mathbf F^{(m+l)}\right),
         \end{gather}
where
\begin{equation}       
\label{eq:Fm}  
        \mathbf F^{(m+l-1)}:=\underbrace{\mathbf F\odot\dots\odot\mathbf F}_{m+l-1},\qquad \mathbf F^{(m+l)}:=\underbrace{\mathbf F\odot\dots\odot\mathbf F}_{m+l}.
\end{equation}        
 If $K=R$ (as in Subsection \ref{subsection:fcr}), then $m=R-K+2=2$, \eqref{eq:FFF} coincides with \eqref{eq:CCC} ($\mathbf F$ coincides with $\mathbf C^{-T}$
 up to column permutation and scaling), and  
   the first phase of Algorithm \ref{alg:onefactormatrix} coincides with steps 1--5 of
          Algorithm \ref{alg:fcr}. For $K<R$ (implying $m>2$) we work as follows.
From Definition \ref{def:matrixtensor} it follows that the  rows of the matrix 
 $\mathbf R_{m,l}(\mathcal T)$ are vectorized versions of   $K\times\dots\times K$  symmetric tensors of order $m+l$. Thus,  in step 2,
  we find the vectors $\mathbf w_1,\dots,\mathbf w_{C^{K-1}_R}$  that form a basis  of the orthogonal complement to $\textup{range}( \mathbf R_{m,l}(\mathcal T)^T)$ in the space $S^{m+l}(\mathbb R^{K^{m+l}})$ (the existence of such a basis
  follows from assumption \textup{(iii)} of Theorem \ref{th:uniquenessandalgorithmrefol}).
   By \eqref{eq:FFF},
   there exists a unique nonsingular $C^{K-1}_R\times C^{K-1}_R$ matrix $\mathbf M$ such that
 \begin{equation}
     \label{eq:WeqFFFM}
     \mathbf W = \mathbf F^{(m+l)}\mathbf M^T.
      \end{equation}
  In step 4, we construct the tensor $\mathcal W$ whose  vectorized frontal slices are the vectors $\mathbf w_1,\dots,\mathbf w_{C^{K-1}_R}$. 
  Reshaping both sides of \eqref{eq:WeqFFFM}  we obtain the CPD $\mathcal W=[\mathbf F, \mathbf F^{(m+l-1)},\mathbf M]_R$ in which the matrices $\mathbf F^{(m+l-1)}$ and $\mathbf M$ have full column rank and $k_{\mathbf F}\geq 2$.
  By  Theorem \ref{theorem:Harshman}, the CPD of $\mathcal W$ can be computed by means of GEVD.
  
  In the second and third phase we use $\mathbf F$ to find $\mathbf A$, $\mathbf B$, $\mathbf C$.
  There are two ways to do this. The first way is to  find $\mathbf C$ from $\mathbf F$ by \eqref{P3}--\eqref{P4} and then to recover $\mathbf A$ and $\mathbf B$ from $\mathcal T$ and $\mathbf C$. The second way is to find $\mathbf A$ and $\mathbf B$ from $\mathcal T$ and $\mathbf F$ and then to recover $\mathbf C$. 
  The second and third phase were thoroughly discussed in \cite{LinkGEVD} and can be taken verbatim from  \cite[Algorithms 1 and 2]{LinkGEVD}.
\begin{example}\label{Example2.12}
Table \ref{table3} contains some examples of CPDs which can be computed by Algorithm 
\ref{alg:onefactormatrix} and cannot be computed by algorithms from \cite{LinkGEVD}. 
The tensors were generated by a PD $[\mathbf A,\mathbf B,\mathbf C]_R$ in which the  entries of
$\mathbf A$, $\mathbf B$, and $\mathbf C$ are independently drawn from the standard normal distribution $N(0,1)$.
We try different values $l=0,1,\dots,$ until condition \textup{(iii)} in Theorem \ref{th:uniquenessandalgorithmrefol} is met (assuming that this will be the case for some $l\geq 0$). 
In our implementation,  we retrieved  the vectors $\mathbf w_1,\dots,\mathbf w_{C^{K-1}_R}$ from the $C^{K-1}_R$-dimensional null space of  a
$C_{R+l+1}^{m+l}\times C_{R+l+1}^{m+l}$ positive  semi-definite  matrix $\mathbf Q$. The storage of $\mathbf Q$ is
the main bottleneck in our implementation.
To give some insight in the  complexity of the algorithm we included the computational time (averaged over 100 random tensors)
and the size of $\mathbf Q$ in the table. 

Uniqueness of the CPDs follows  from Theorem \ref{th:uniquenessandalgorithmrefol}.
By comparison, the results of \cite{PartII} guarantee uniqueness only for rows 1--4 (see \cite[Table 3.1]{PartII}).

\begin{table}[!h]
   \begin{center}
 \caption{ Upper bounds on $R$ under which the CPD of a generic $I\times J\times K$ tensor
 can be computed by Algorithm \ref{alg:onefactormatrix} (see Example \ref{Example2.12} details).}
  \begin{tabular}{|c|c|c|c|c|c|}
 \hline
 dimensions of $\mathcal T$&    $R$&  $m$&  $l$& $C_{R+l+1}^{m+l}$ &  computational time (sec)   \\
 \hline
  $4\times 5\times 6$       & 7  & 3 & 1 & 126& 0.182 	\\
  $5\times 7\times 7$       & 9  & 4 & 1 & 462& 1.598 	\\
 $6 \times 9 \times 8$     & 11 & 5 & 1 & 1716& 28.616    \\       
 $7 \times 7 \times 7$     & 10 & 5 & 1 & 924& 8.192\\ 
 $4 \times 6 \times 8$     & 9 & 3 & 1 & 330& 0.63 	\\ 
 $4 \times 7 \times 10$     & 11 & 3 & 1 & 715& 2.352 	\\ 
 $5 \times 6 \times 6$     & 8 & 4 & 2 & 462 & 1.256	\\ 
 $5 \times 7 \times 8$     & 10 & 4 & 2 &  1716& 14.552 \\ 
  \hline
 \end{tabular}
 \label{table3}
 \end{center}
 \end{table}
\end{example}
 \section{Proofs related to Sections \ref{subsection:fcr} and \ref{subsectio:notfcr}}\label{proofsnfcr}
    In this section  we 1) prove Theorems \ref{th:uniquenessonefm1} and \ref{th:uniquenessonefm2}; 2)  show that the assumptions in  Theorem \ref{th:uniquenessandalgorithm} are more restrictive than the assumptions
  in Theorem \ref{theorem2.9}, which implies the statement on uniqueness in Theorem \ref{th:uniquenessandalgorithm};
 3) prove that assumption \textup{(iii)} in Theorem \ref{th:uniquenessandalgorithm} is equivalent to assumption \textup{(iii)} in Theorem \ref{th:uniquenessandalgorithmrefol}; 4) prove statements \eqref{eq:beforeFFF}--\eqref{eq:FFF}; 
 5) prove Theorem \ref{th:maintheoremfcr}.
    \subsection{Proofs of  Theorems \ref{th:uniquenessonefm1} and \ref{th:uniquenessonefm2}}\label{subsection6.1}
In the sequel, $\omega(\lambda_1, \dots,\lambda_R)$ denotes the number of nonzero entries of $[\lambda_1\ \dots\ \lambda_R]^T$.
The following  condition \condWm was introduced in \cite{PartI, PartII} in terms of $m$-th compound matrices.
In this paper we will use the following (equivalent)  definition of \condWm.

\begin{definition}\label{Def:Wm}
We say that  condition \condWm  holds for the triplet of matrices $(\mathbf A,\mathbf B,\mathbf C)\in \mathbb R^{I\times R}\times \mathbb R^{J\times R}
\times \mathbb R^{K\times R}$ if $\omega(\lambda_1, \dots,\lambda_R)\leq m-1$ whenever
 \begin{gather}
r_{\mathbf A\Diag{\lambda_1,\dots,\lambda_R}\mathbf B^T}\leq m-1\quad  \text{ and }\quad [\lambda_1\ \dots\ \lambda_R]^T\in\textup{range}(\mathbf C^T)\label{eqWm2new}.
\end{gather}
\end{definition}

Since the rank of the product $\mathbf A\Diag{\lambda_1,\dots,\lambda_R}\mathbf B^T$ does not exceed the rank of the factors and
$r_{\Diag{\lambda_1,\dots,\lambda_R}}=\omega(\lambda_1,\dots,\lambda_R)$, we always have the implication
\begin{equation}\label{eqdirectimplication}
\omega(\lambda_1,\dots,\lambda_R)\leq m-1\quad\Rightarrow\quad r_{\mathbf A\Diag{\lambda_1,\dots,\lambda_R}\mathbf B^T}\leq m-1.
\end{equation}
By Definition \ref{Def:Wm}, condition \condWm  holds for the triplet $(\mathbf A,\mathbf B,\mathbf C)$ if and only if  the opposite implication  in \eqref{eqdirectimplication} holds for all $[\lambda_1\ \dots\ \lambda_R]\in\textup{range}(\mathbf C^T)\subset\mathbb R^R$.

The following  results on rank and uniqueness of one factor matrix  have been obtained in \cite{PartI}.
\begin{proposition}\label{prmostgeneraldis} (see \cite[Proposition 4.9]{PartI})
 Let $\mathcal T=(t_{ijk})_{i,j,k=1}^{I,J,K}=[\mathbf A,\mathbf B,\mathbf C]_R$, $r_{\mathbf C}=K\leq R$.
Assume that
\begin{itemize}
\item[(i)] $k_{\mathbf C}\geq 1$;
\item[(ii)] $\mathbf A\odot\mathbf B$ has full column rank;
\item[(iii)]
conditions $\text{\textup{(W{\scriptsize m})}},\dots,\text{\textup{(W{\scriptsize 1})}}$ hold for the triplet of matrices $(\mathbf A,\mathbf B,\mathbf C)$.
\end{itemize}
Then $r_{\mathcal T}=R$ and the third factor matrix  of $\mathcal T$  is unique.
\end{proposition}

\begin{proposition}\label{prmostgeneraldis2}(see \cite[Corollary 4.10]{PartI})
 Let $\mathcal T=(t_{ijk})_{i,j,k=1}^{I,J,K}=[\mathbf A,\mathbf B,\mathbf C]_R$, $r_{\mathbf C}=K\leq R$.
Assume that
\begin{itemize}
\item[(i)] $k_{\mathbf C}\geq 1$;
\item[(ii)] $\mathbf A\odot\mathbf B$ has full column rank;
\item[(iii)] $\min (k_{\mathbf A},k_{\mathbf B})\geq m-1$;
\item[(iv)]
condition $\text{\textup{(W{\scriptsize m})}}$ holds  for the triplet of matrices $(\mathbf A,\mathbf B,\mathbf C)$.
\end{itemize}
Then $r_{\mathcal T}=R$ and the third factor matrix of $\mathcal T$  is unique.
\end{proposition}

One can easily notice the similarity between the assumptions in Theorems \ref{th:uniquenessonefm1}--\ref{th:uniquenessonefm2}
and the assumptions in Propositions \ref{prmostgeneraldis}--\ref{prmostgeneraldis2}. 
The proofs of Theorems \ref{th:uniquenessonefm1}--\ref{th:uniquenessonefm2}  follow from 
Propositions \ref{prmostgeneraldis}--\ref{prmostgeneraldis2} and the following lemma.
\begin{lemma}\label{some:lemma}
Let $\mathbf A\in\mathbb R^{I\times R}$, $\mathbf B\in\mathbb R^{J\times R}$, and $\mathbf C\in\mathbb R^{K\times R}$, $r_{\mathbf C}=K\leq R$, $k\leq m=R-K+2$, and let $l$
  be a nonnegative integer. Let also
  the matrix  $\mathbf \Phi_{k,l}(\mathbf A,\mathbf B)$
  be defined as in Definition \ref{def:matrixPhi}, the matrix 
  $\mathbf S_{k+l}(\mathbf C)$ be defined as in Definition \ref{def:matrixS}, and
  $\mathbf U$ be a matrix such that its columns  form a basis for $\textup{range}(\mathbf S_{k+l}(\mathbf C)^T)$.
   Assume that
  \begin{equation}
 \text{the matrix } \mathbf \Phi_{k,l}(\mathbf A,\mathbf B)\mathbf U \text{ has full column rank.}
 \label{eq:newWk}
 \end{equation}
 Then condition $\text{\textup{(W{\scriptsize k})}}$ holds  for the triplet of matrices $(\mathbf A,\mathbf B,\mathbf C)$.
\end{lemma}
\begin{proof}
Let \eqref{eqWm2new} hold for $m=k$. We need to show that $\omega(\lambda_1, \dots,\lambda_R)\leq k-1$.
Since
$[\lambda_1\ \dots\ \lambda_R]^T\in\textup{range}(\mathbf C^T)$ and $r_{\mathbf C}=K$, there exists a unique vector $\mathbf x\in\mathbb R^K$ such that $[\lambda_1\ \dots\ \lambda_R]=\mathbf x^T\mathbf C$. Hence, we need to show that $\mathbf x$ is orthogonal to at least
$R-k+1$ columns of $\mathbf C$.

By \eqref{eqWm2new}, there exist $\tilde{\mathbf A}\in\mathbb R^{I\times R}$ and $\tilde{\mathbf B}\in\mathbb R^{J\times R}$ such that
\begin{equation}
\mathbf A\Diag{\lambda_1,\dots,\lambda_R}\mathbf B^T = \tilde{\mathbf A}\tilde{\mathbf B}^T\
\label{eq:Wwwmatr}
\end{equation}
and $\max(r_{\tilde{\mathbf A}},r_{\tilde{\mathbf B}})\leq k-1$.
Since $\mathbf a\mathbf b^T\lambda = \mathbf a\outerprod\mathbf b\outerprod\lambda$, we can consider \eqref{eq:Wwwmatr}  as an equality of two PDs of an
$I\times J\times 1$ tensor
\begin{equation*}
\sum\limits_{r=1}^R\mathbf a_r\outerprod\mathbf b_r\outerprod\lambda_r=\sum\limits_{r=1}^R\tilde{\mathbf a}_r\outerprod\tilde{\mathbf b}_r\outerprod 1.
\end{equation*}
Hence, by \eqref{eq:mainidentity},
\begin{equation}
\begin{split}
{\mathbf\Phi}_{k,l}(\mathbf A,\mathbf B)\mathbf S_{k+l}(\mathbf x^T\mathbf C)^T =&\
{\mathbf\Phi}_{k,l}(\mathbf A,\mathbf B)\mathbf S_{k+l}([\lambda_1\ \dots\ \lambda_R])^T =\\
&{\mathbf\Phi}_{k,l}(\tilde{\mathbf A},\tilde{\mathbf B})\mathbf S_{k+l}([1\ \dots\ 1])^T.
\end{split}
\label{eq:Phikkkll}
\end{equation}
Since $\max(r_{\tilde{\mathbf A}},r_{\tilde{\mathbf B}})\leq k-1$, it follows from Definition \ref{def:matrixPhi} that
${\mathbf\Phi}_{k,l}(\tilde{\mathbf A},\tilde{\mathbf B})$ is the zero matrix (cf. explanation at the end of Section \ref{subsection:construction}).
Besides, it easily follows from Definition
\ref{def:matrixS} that
$$
\mathbf S_{k+l}(\mathbf C)^T(\underbrace{\mathbf x\otimes\dots\otimes\mathbf x}_{k+l})=\mathbf S_{k+l}(\mathbf x^T\mathbf C)^T.
$$
Thus, \eqref{eq:Phikkkll} takes the form
$$
{\mathbf\Phi}_{k,l}(\mathbf A,\mathbf B)\mathbf S_{k+l}(\mathbf C)^T(\mathbf x\otimes\dots\otimes\mathbf x)=
{\mathbf\Phi}_{k,l}(\mathbf A,\mathbf B)\mathbf S_{k+l}(\mathbf x^T\mathbf C)^T=\mathbf 0.
$$
Hence, by \eqref{eq:newWk}, the vector $\mathbf x\otimes\dots\otimes\mathbf x$ is orthogonal to the range of
$S_{k+l}(\mathbf C)$. In particular,
\begin{equation*}
\begin{split}
(\mathbf x\otimes\dots\otimes\mathbf x)^T
\sum\limits_{\substack{(s_1,\dots,s_{k+l})\in\\ P_{\{r_1,\dots,r_k,\dots,r_k\}}}}\mathbf c_{s_1}\otimes\dots\otimes\mathbf c_{s_{m+l}}&=\\
(\mathbf x^T\mathbf c_{r_1})\cdots (\mathbf x^T\mathbf c_{r_{k-1}})(\mathbf x^T\mathbf c_{r_k})^{l+1}&=0
\end{split}
\end{equation*}
for all $(k+l)$-tuples $(r_1,\dots,r_k,\dots,r_k)$ such that $1\leq r_1<\dots<r_k\leq R$, yielding that
 $\mathbf x$ is orthogonal to at least
$R-k+1$ columns of $\mathbf C$.
\end{proof}
\subsection{Proof of statement on rank and uniqueness in  Theorem \ref{th:uniquenessandalgorithm}}
In Lemma \ref{lemma:onemorelemma} below we prove that $\min(k_{\mathbf A},k_{\mathbf B})\geq m$. It is clear that  condition 
$\min(k_{\mathbf A},k_{\mathbf B})\geq m$  and assumption \textup{(i)} in Theorem \ref{th:uniquenessandalgorithm} imply   assumption \textup{(iii)} in Theorem \ref{th:uniquenessonefm2} and condition \eqref{eq:uniqviaonefm}. Hence, by Theorem \ref{theorem2.9}, $r_{\mathcal T} = R$ and   the CPD of tensor $\mathcal T$ is unique.
\begin{lemma}\label{lemma:onemorelemma}
Let assumptions  \textup{(i)} and \textup{(iii)} in Theorem \ref{th:uniquenessandalgorithm} hold. Then\\
$\min(k_{\mathbf A},k_{\mathbf B})\geq m$.
\end{lemma}
\begin{proof}
Assume to the contrary that $k_{\mathbf A}< m$ or $k_{\mathbf B}< m$.
W.l.o.g. we assume that the first $m$ columns of $\mathbf A$ are linearly dependent.
We will get a contradiction with assumption \textup{(iii)} by constructing a nonzero vector $\mathbf f\in  \textup{range}(\mathbf S_{m+l}(\mathbf C)^T)$ such that ${\mathbf\Phi}_{m,l}(\mathbf A,\mathbf B)\mathbf f=\mathbf 0$.
Since $k_{\mathbf C}=K$, there exists  $\mathbf x\in\mathbb R^K$ such that
\begin{equation}
\label{eq:constructionofx}
\mathbf x^T\mathbf c_1\ne 0,\dots, \mathbf x^T\mathbf c_m\ne 0,\ \ \mathbf x^T\mathbf c_{m+1} =\dots=\mathbf x^T\mathbf c_R=0.
\end{equation}
We set $\mathbf f=\mathbf S_{m+l}(\mathbf C)^T(\underbrace{\mathbf x\otimes\dots\otimes\mathbf x}_{m+l})$ and we index the entries of $\mathbf f$ by $(m+l)$-tuples as in \eqref{eq:mplusltuples}. One can easily show that $\mathbf f$  has entries $(\mathbf x^T\mathbf c_{r_1})\dots (\mathbf x^T\mathbf c_{r_{m+l}})$. Hence, by \eqref{eq:constructionofx},
\begin{align*}
(\mathbf x^T\mathbf c_{r_1})\dots (\mathbf x^T\mathbf c_{r_{m+l}})=0,\ &\text{ if }\ 
\{r_1,\dots,r_{m+l}\}\setminus\{1,\dots,m\}\ne \emptyset,\\
(\mathbf x^T\mathbf c_{r_1})\dots (\mathbf x^T\mathbf c_{r_{m+l}})\ne0,\ &\text{ if }\ 
\{r_1,\dots,r_{m+l}\}\setminus\{1,\dots,m\}= \emptyset.
\end{align*}
On the other hand, by Definition \ref{def:matrixPhi} and the assumption of linear dependence of the vectors $\mathbf a_1,\dots,\mathbf a_m$, the columns of ${\mathbf\Phi}_{m,l}(\mathbf A,\mathbf B)$ indexed by the $(m+l)$-tuples \eqref{eq:mplusltuples} such that  
$
\{r_1,\dots,r_{m+l}\}\setminus\{1,\dots,m\}=\emptyset
$
are zero. Hence, ${\mathbf\Phi}_{m,l}(\mathbf A,\mathbf B)\mathbf f=\mathbf 0$.
\end{proof}
\subsection{Properties of the matrix $\mathbf S_{m+l}(\mathbf C)^T$}
The following auxiliary Lemma will be used in Subsections \ref{Subsectionequiv} and \ref{Subsection:last}.
Since the proof is rather long and technical, it is included in \ref{section:Appendix}. 
\begin{lemma}\label{lemma:suppl}
Let $\mathbf C\in\mathbb R^{K\times R}$, $k_{\mathbf C}=K$, $m=R-K+2$, $l\geq 0$,  let $\mathbf F$  satisfy \eqref{P1}--\eqref{P2}, and let $\mathbf F^{(m+l)}$ be defined by \eqref{eq:Fm}. Then
\begin{itemize}
\item[\textup{(i)}] $\dim  \left(\ker (\mathbf S_{m+l}(\mathbf C)^T)\bigcap S^{m+l}(\mathbb R^{K^{m+l}})\right)=C^{K-1}_R$;
\item[\textup{(ii)}] $\ker (\mathbf S_{m+l}(\mathbf C)^T)\bigcap S^{m+l}(\mathbb R^{K^{m+l}})=\textup{range}\left(\mathbf F^{(m+l)}\right)$;					
\item[\textup{(iii)}]$\textup{range} (\mathbf S_{m+l}(\mathbf C)^T)=\mathbf S_{m+l}(\mathbf C)^T(S^{m+l}(\mathbb R^{K^{m+l}}))$;
\item[\textup{(iv)}]
$\dim\ (\textup{range} (\mathbf S_{m+l}(\mathbf C)^T))=
C^{m+l}_{K+m+l-1}-C^{K-1}_R$.
\end{itemize}
\end{lemma}
\subsection{Proof of equivalence of Theorems \ref{th:uniquenessandalgorithm} and \ref{th:uniquenessandalgorithmrefol}}
\label{Subsectionequiv}
 We prove that assumption \textup{(iii)} in Theorem \ref{th:uniquenessandalgorithm} is equivalent to
assumption \textup{(iii)} in Theorem \ref{th:uniquenessandalgorithmrefol}.
 By \eqref{equation:2.19},
it is sufficient to prove that
\begin{equation}
      \begin{split}
      \label{eq4.8}
  \dim  \left(\ker (\mathbf R_{m,l}(\mathcal T))\bigcap S^{m+l}(\mathbb R^{K^{m+l}})\right)\geq C^{K-1}_R+1\Leftrightarrow\\
\text{the matrix }\  \mathbf \Phi_{m,l}(\mathbf A,\mathbf B)\mathbf U_m\ \text{does not have full column rank}.
  \end{split}
 \end{equation}
To prove \eqref{eq4.8} we will use the following result for 
$\mathbf X:=\mathbf \Phi_{m,l}(\mathbf A,\mathbf B)$, $\mathbf Y:=\mathbf S_{m+l}(\mathbf C)^T$, and 
$E:=S^{m+l}(\mathbb R^{K^{m+l}})$: if  $E$ is a subspace and $\mathbf X$ and $\mathbf Y$ are matrices such that $\mathbf X\mathbf Y$ is defined, then
 \begin{equation}
 \label{eq:4.9}
      \begin{split}
      &\dim\left(\textup{ker}(\mathbf X\mathbf Y)\cap E \right)\geq  \dim\left(\textup{ker}(\mathbf Y)\cap E \right)+1\quad \Leftrightarrow\\
      &\text{there exists a nonzero vector}\ \mathbf f\in   E\setminus \textup{ker}(\mathbf Y)\ \text{such that }\ \mathbf X\mathbf Y\mathbf f=0.
      \end{split}
 \end{equation}
 We have
       \begin{flalign*}
          \dim  \left(\ker (\mathbf R_{m,l}(\mathcal T))\bigcap S^{m+l}(\mathbb R^{K^{m+l}})\right)\geq C^{K-1}_R+1
          \qquad\qquad\qquad\qquad\xLeftrightarrow{\eqref{eq:mainidentity}}\\
         \left\{\begin{aligned}
            \dim\left(\text{ker} ({\mathbf\Phi}_{m,l}(\mathbf A,\mathbf B)\mathbf S_{m+l}(\mathbf C)^T)\bigcap S^{m+l}(\mathbb R^{K^{m+l}})  \right) \geq C^{K-1}_R+1=\\
           \dim\left(\text{ker} (\mathbf S_{m+l}(\mathbf C)^T)\bigcap S^{m+l}(\mathbb R^{K^{m+l}})  \right)+1
          \end{aligned}\right\}  \xLeftrightarrow{\eqref{eq:4.9}}\\
          \left\{\begin{aligned} 
          \text{there exists a nonzero vector}\  \mathbf f\in   S^{m+l}(\mathbb R^{K^{m+l}})\setminus \textup{ker}(\mathbf S_{m+l}(\mathbf C)^T) \\
           \text{ such that } \ \mathbf \Phi_{m,l}(\mathbf A,\mathbf B)\mathbf S_{m+l}(\mathbf C)^T\mathbf f=0
                    \end{aligned}\right\}\ \xLeftrightarrow{\ \ \ \ }\\
          \text{the matrix }\  \mathbf \Phi_{m,l}(\mathbf A,\mathbf B)\mathbf U_m\ \text{does not have full column rank},
            \end{flalign*}
   where the equality in the second statement holds by Lemma \ref{lemma:suppl} \textup{(i)} and the last equivalence follows from $\textup{range}(\mathbf U_m)=\textup{range}(\mathbf S_{m+l}(\mathbf C)^T)$.
 \subsection{Proof of the statement on algebraic computation in  Theorem \ref{th:uniquenessandalgorithm}}\label{Subsection:last}
  The overall procedure that constitutes the proof of the statement on algebraic computation is summarized
 in Algorithm \ref{alg:onefactormatrix}  and  explained in Subsection \ref{subsection2.4.2}.
 In this subsection we prove statements \eqref{eq:beforeFFF}--\eqref{eq:FFF}.
 \begin{lemma}\label{lemma4.5}
Let assumptions \textup{(i)} and \textup{(iii)} in Theorem \ref{th:uniquenessandalgorithmrefol} hold
and let $\mathbf F$ satisfy \eqref{P1}--\eqref{P2}.
Then \eqref{eq:beforeFFF}--\eqref{eq:FFF} hold.
\end{lemma}
\begin{proof} 
The implication $k_{\mathbf C}=K\ \Rightarrow\ $ \eqref{eq:beforeFFF} was proved in \cite[Proposition 1.10]{LinkGEVD}.  
In Subsection \ref{Subsectionequiv} we proved that
assumption \textup{(iii)} in Theorem \ref {th:uniquenessandalgorithm} holds. By \eqref{eq:mainidentity}, 
Theorem \ref {th:uniquenessandalgorithm} \textup{(iii)}, and Lemma \ref{lemma:suppl} we have
\begin{equation*}
\begin{split}
&\ker (\mathbf R_{m,l}(\mathcal T))\bigcap S^{m+l}(\mathbb R^{K^{m+l}})=\\
&\text{ker} ({\mathbf\Phi}_{m,l}(\mathbf A,\mathbf B)\mathbf S_{m+l}(\mathbf C)^T)\bigcap S^{m+l}(\mathbb R^{K^{m+l}})=\\
&\text{ker} (\mathbf S_{m+l}(\mathbf C)^T)\bigcap S^{m+l}(\mathbb R^{K^{m+l}}) = 
\textup{range}\left(\mathbf F^{(m+l)}\right),
\end{split}
\end{equation*}
which completes the proof of \eqref{eq:FFF}.
\end{proof}
\subsection{Proof of  Theorem \ref{th:maintheoremfcr}}\label{subsection6.6}We check the assumptions in 
Theorem \ref{th:uniquenessandalgorithmrefol} for $m=2$. 
Assumption \textup{(i)} holds since $r_{\mathbf C}=K=R$ implies $k_{\mathbf C}=K$ and 
assumption \textup{(iii)} coincides with \eqref{eq:maincondfcr}. 
To prove assumption \textup{(ii)} we assume to the contrary that $(\mathbf A\odot\mathbf B)[\lambda_1\ \dots\ \lambda_R]^T=\mathbf 0$. Then  $r_{\mathbf A\Diag{\lambda_1,\dots,\lambda_R} \mathbf B^T)}=0$.
In Subsection \ref{Subsectionequiv}  we explained that assumption \textup{(iii)} in Theorem \ref {th:uniquenessandalgorithm} also holds. Hence, by Lemma \ref{some:lemma},
condition $\text{\textup{(W{\scriptsize 2})}}$ holds  for the triplet $(\mathbf A,\mathbf B,\mathbf C)$.
Hence, at most one of the values $\lambda_1,\dots,\lambda_R$ is not zero. If such a $\lambda_r$ exists, then
$\mathbf a_r=\mathbf 0$ or $\mathbf b_r=\mathbf 0$ yielding that $\min(k_{\mathbf A}, k_{\mathbf B})=0$. On the other hand,
 by Lemma \ref{lemma:onemorelemma}, $\min(k_{\mathbf A}, k_{\mathbf B})\geq 2$,  which is a contradiction.
 Hence, $\lambda_1=\dots=\lambda_R=0$.
 
\section{Discussion}\label{sec:discussion}
A number of conditions (called $\text{\textup{(K{\scriptsize m})}}$, $\text{\textup{(C{\scriptsize m})}}$, 
$\text{\textup{(U{\scriptsize m})}}$, and $\text{\textup{(W{\scriptsize m})}}$) for uniqueness of CPD of a specific tensor have been proposed in \cite{PartI,PartII}. It was  shown that 
each subsequent condition in $\text{\textup{(K{\scriptsize m})}},\dots, \text{\textup{(W{\scriptsize m})}}$ is
more general than the preceding one, but harder to use. Verification of conditions $\text{\textup{(K{\scriptsize m})}}$ 
and $\text{\textup{(C{\scriptsize m})}}$ reduces to the computation of matrix rank. In contrast,  conditions $\text{\textup{(U{\scriptsize m})}}$ and $\text{\textup{(W{\scriptsize m})}}$  are  not easy to check for a specific tensor
but hold automatically for  generic tensors of certain dimensions and rank \cite{AlgGeom1}.

In this paper we have proposed new sufficient conditions for uniqueness  that
can be verified by  the computation of matrix rank,  are more relaxed
than $\text{\textup{(K{\scriptsize m})}}$ and $\text{\textup{(C{\scriptsize m})}}$,
but that cannot be more relaxed than $\text{\textup{(W{\scriptsize m})}}$.
Nevertheless, examples illustrate  that in many cases  the new conditions may be considered as an ``easy to check analogue''
of $\text{\textup{(U{\scriptsize 2})}}$ ($\Leftrightarrow\text{\textup{(W{\scriptsize 2})}}$) and $\text{\textup{(W{\scriptsize m})}}$.

We have also  proposed an  algorithm to compute the factor matrices. The algorithm relies only on standard linear algebra, and has as input the tensor $\mathcal T$, the tensor rank $R$, and  a nonnegative integer parameter $l$. 
 The algorithm basically reduces the problem to the construction of a $C^{m+l}_{K+m+l-1}\times C^{m+l}_{K+m+l-1}$ matrix $\mathbf Q$, the computation of its $C^{K-1}_R$-dimensional null space, and  the GEVD of a $C^{K-1}_R\times C^{K-1}_R$ matrix pencil, where  $m=R-K+2$.
For $l=0$, Algorithms \ref{alg:fcr}   and \ref{alg:onefactormatrix} coincide with algorithms from \cite{DeLathauwer2006} and \cite{LinkGEVD}, respectively.
Our derivation is different from the derivations in \cite{DeLathauwer2006} and \cite{LinkGEVD} but has the same structure: from the CPD $\mathcal T=[\mathbf A,\mathbf B,\mathbf C]_R$ we derive
a set of equations that depend only on $\mathbf C$;  we find $\mathbf C$ from the new system by  means of GEVD,
and then recover $\mathbf A$ and $\mathbf B$ from $\mathcal T$ and  $\mathbf C$.

It is interesting to note that the new algorithm (with $l=1$) computes the CPD of a generic $3\times 7\times 12$ tensor of rank $12$ 
in less than $1$ second while optimization-based algorithms  (we checked a Gauss-Newton  dogleg trust region method)
fail to find the solution in a reasonable amount of time.

We have demonstrated that our algorithm (with $l\leq 2$) can find the  CPD of a generic $I\times J\times K$ tensor  of rank $R$
if $R\leq K\leq (I-1)(J-1)$ and $R\leq 24$. We conjecture that the  algorithm (possibly with $l\geq 3$) can also find the CPD for $R\geq 25$.
(It is known that  the CPD of a generic tensor is not unique if  $R> (I-1)(J-1)$).
In that case the  $C^{m+l}_{K+m+l-1}\times C^{m+l}_{K+m+l-1}$ matrix $\mathbf Q$  becomes large and the computation,
as it is proposed in the paper, becomes infeasible.
Since the null space of  $\mathbf Q$ is just $R$-dimensional the approach may possibly be scaled by using 
iterative methods to compute the null space. 

\appendix
\section{Derivation of identity \eqref{eq:mainidentity}}\label{subsection:construction}
 Let $\mathcal T=(t_{ijk})_{i,j,k=1}^{I,J,K}=[\mathbf A,\mathbf B,\mathbf C]_R$. In this section we establish a link between the matrix 
  $\mathbf R_{m,l}(\mathcal T)$ defined in subsection \ref{subsubsection:mainconstruction}
  and the factor matrices $\mathbf A$, $\mathbf B$, and $\mathbf C$.
We show that the matrix $\mathbf R_{m,l}(\mathcal T)$ is obtained from $\mathcal T$ by  taking the following steps: 1) taking the $(m+l)$th Kronecker power of $\mathcal T$; 2) making two  partial skew-symmetrizations and one partial symmetrization of the result; 3) reshaping the result into an $I^{m+l}J^{m+l}\times K^{m+l}$ matrix. The main identity is obtained by applying steps 1)--3) to the both sides of \eqref{eqintro2}.
\subsection{Step 1: Kronecker product power of $\mathcal T$}
 The Kronecker product square of $\mathcal T$,
 $\mathcal T^{(2)}:=\mathcal T\otimes\mathcal T$, is an $I\times J\times K$ block-tensor whose  
 $(i,j,k)$th block is the $I\times J\times K$ tensor $t_{ijk}\mathcal T$. Equivalently,
 $\mathcal T^{(2)}$  is an $I^2\times J^2\times K^2$ tensor whose
 $(\tilde i,\tilde j,\tilde k):=(i_1-1)I+i_2,(j_1-1)J+j_2,(k_1-1)K+k_2)$th entry is
 $$
 t^{(2)}_{\tilde i\tilde j\tilde k} = t_{i_1j_1k_1}t_{i_2j_2k_2}.
 $$
 Similarly, the $(l+m)$-th Kronecker product power of $\mathcal T$,
 \begin{equation*}
 \mathcal T^{(m+l)} := \underbrace{\mathcal T\otimes\dots\otimes\mathcal T}_{m+l},
 \end{equation*} 
is an  
$I^{m+l}\times J^{m+l}\times K^{m+l}$ tensor  whose $(\tilde i,\tilde j,\tilde k)$th entry is
\begin{equation}
t^{(m+l)}_{\tilde i\tilde j\tilde k}=
t_{i_1j_1k_1}t_{i_2j_2k_2}\cdots t_{i_{m+l}j_{m+l}k_{m+l}},
\label{eq:3rdoversion}
\end{equation}
where $\tilde i$ ,$\tilde j$, and $\tilde k$ are defined in \eqref{eq:tildei}, \eqref{eq:tildej}, and \eqref{eq:tildek}, respectively.
One can easily check that if $\mathcal T=\sum\limits_{r=1}^R\mathbf a_r\outerprod \mathbf b_r\outerprod \mathbf c_r$, then
 $$
\mathcal T^{(m+l)} = 
\sum\limits_{r_1,\dots,r_{m+l}=1}^R
(\mathbf a_{r_1}\otimes\dots\otimes\mathbf a_{r_{m+l}})\outerprod
(\mathbf b_{r_1}\otimes\dots\otimes\mathbf b_{r_{m+l}})\outerprod
(\mathbf c_{r_1}\otimes\dots\otimes\mathbf c_{r_{m+l}}).
$$
\subsection{Step 2: two  partial skew-symmetrizations and one partial symmetrization of a reshaped version of $\mathcal T^{(m+l)}$}
Recall that a higher-order tensor is said to be {\em symmetric  (resp. skew-symmetric) with respect to a given group of indices} or {\em partially symmetric (resp. skew-symmetric)} if its coordinates
do not alter by an arbitrary permutation of these indices (resp. if the sign changes with every interchange of two arbitrary indices in the group).

Let us recall the operations of (complete) symmetrization and skew-symmetriza\-tion.
With a general  $k$th-order $L\times \dots\times L$ tensor $\mathcal N$ one can  associate
its symmetric part $S^k (\mathcal N)$ and skew-symmetric part  $\Lambda^k(\mathcal N)$ as follows.
By construction,  $S^k (\mathcal N)$ is a tensor whose entry with indices $l_1,\dots,l_k$
is equal to
\begin{equation}
\frac{1}{k!}\sum\limits_{(p_1,\dots,p_k)\in P_{\{l_1,\dots,l_k\}}}  n_{p_1\dots p_k}.
\label{symmetrization}
\end{equation}
That is, to get $S^k (\mathcal N)$ we should take the  average  of $k!$ tensors obtained from $\mathcal N$ by
all possible permutations of the indices. Similarly,
$\Lambda^k(\mathcal N)$ is a tensor whose entry with indices $l_1,\dots,l_k$ is equal to
\begin{equation}
\begin{cases}
\frac{1}{k!}\sum\limits_{(p_1,\dots,p_k)\in P_{\{l_1,\dots,l_k\}}} \sgn(p_1,\dots,p_k)  n_{p_1\dots p_k},& \text{if } l_1,\dots,l_k\ \text{are distinct},\\
0,&\text{otherwise,}
\end{cases}
\label{skewsymmetrization}
\end{equation}
where 
$$
\sgn(p_1,\dots,p_k) \text{ denotes the signature of the permutation } (p_1,\dots,p_k).
$$
The definition of  $\Lambda^k(\mathcal N)$  differs from that of $S^k (\mathcal N)$  in that the signatures of the permutations are taken into account and that the entries of $\Lambda^k(\mathcal N)$ with  repeated indices  are necessarily zeros.
One can easily check that if $\mathcal N= \mathbf d_1\outerprod\dots\outerprod \mathbf d_k$ (that is, $\mathcal N$ is the $k$th order rank-$1$ tensor), then
\begin{align}
 S^k(\mathbf d_1\outerprod\dots\outerprod \mathbf d_k) &= 
\frac{1}{k!}\sum\limits_{(p_1,\dots,p_k)\in P_{\{1,\dots,k\}}}  \mathbf d_{p_1}\outerprod\dots\outerprod \mathbf d_{p_k},\label{eq:jtheq0}\\
 \Lambda^k(\mathbf d_1\outerprod\dots\outerprod \mathbf d_k) &= 
\frac{1}{k!}\sum\limits_{(p_1,\dots,p_k)\in P_{\{1,\dots,k\}}}  \sigma(p_1,\dots,p_k)\mathbf d_{p_1}\outerprod\dots\outerprod \mathbf d_{p_k}.
\label{eq:jtheq1}
\end{align}
Partial (skew-)symmetrization is a (skew-)symmetrization with respect to a given group of indices. Instead of presenting the formal definitions we illustrate both notions for an $M \times L\times L$ tensor $\mathcal N=(n_{ml_1l_2})_{m,l_1,l_2=1}^{M,L,L}$.
Partial symmetrization with respect to the group of indices $\{2,3\}$ maps the tensor $\mathcal N$ to a tensor that we denote by
$(\mathbf I_M\outerprod S^2)\mathcal N$, whose entry with indices $(m,l_1,l_2)$ is equal to
$$
\sum\limits_{(p_1,p_2)\in P_{\{l_1,l_2\}}} n_{ml_1l_2} = n_{ml_1l_2}+n_{ml_2l_1}.
$$
Similarly, by $(\mathbf I_M\outerprod \Lambda^2)\mathcal N$ we denote the tensor whose entry with indices $(m,l_1,l_2)$ is equal to
$$
\begin{cases}
\sum\limits_{(p_1,p_2)\in P_{\{l_1,l_2\}}} \sgn(p_1,p_2)n_{ml_1l_2} = n_{ml_1l_2}-n_{ml_2l_1},&\text{if } \ l_1\ne l_2,\\
0,&\text{if } \ l_1= l_2.
\end{cases}
$$
If $\mathcal N=\mathbf d_1\outerprod\mathbf d_2\outerprod\mathbf d_3\in\mathbb R^{M\times L\times L}$, then
\begin{align}
(\mathbf I_M\outerprod S^2)(\mathbf d_1\outerprod\mathbf d_2\outerprod\mathbf d_3)=
\mathbf d_1\outerprod S^2(\mathbf d_2\outerprod\mathbf d_3)=
\mathbf d_1\outerprod\mathbf d_2\outerprod\mathbf d_3+\mathbf d_1\outerprod\mathbf d_3\outerprod\mathbf d_2, \label{eq:s2MLL}\\
(\mathbf I_M\outerprod \Lambda^2)(\mathbf d_1\outerprod\mathbf d_2\outerprod\mathbf d_3)=
\mathbf d_1\outerprod \Lambda^2(\mathbf d_2\outerprod\mathbf d_3)=\mathbf d_1\outerprod\mathbf d_2\outerprod\mathbf d_3-\mathbf d_1\outerprod\mathbf d_3\outerprod\mathbf d_2.
\label{eq:l2MLL}
\end{align}
Thus, operations $(\mathbf I_M\outerprod S^2)$ and $(\mathbf I_M\outerprod \Lambda^2)$ symmetrize and skew-symmetrize the horizontal
slices of $\mathcal N$.

 Let us reshape the tensor $\mathcal T^{(m+l)}$ into an $I\times\dots\times I\times J\times\dots\times J\times K\times\dots\times K$ (each letter is repeated $m+l$ times) tensor $\widehat{\mathcal T}^{(m+l)}$ as:
\begin{equation}
\widehat{\mathcal T}^{(m+l)} = 
\sum\limits_{r_1,\dots,r_{m+l}=1}^R
(\mathbf a_{r_1}\outerprod\dots\outerprod\mathbf a_{r_{m+l}})\outerprod
(\mathbf b_{r_1}\outerprod\dots\outerprod\mathbf b_{r_{m+l}})\outerprod
(\mathbf c_{r_1}\outerprod\dots\outerprod\mathbf c_{r_{m+l}}).\label{eq:2.6}
\end{equation}
Then the  entries of $\widehat{\mathcal T}^{(m+l)}$ are given by 
\begin{equation}
\widehat{t}^{(m+l)}_{i_1\dots i_{m+l} j_1 \dots j_{m+l} k_1 \dots k_{m+l}}=
t_{i_1j_1k_1}t_{i_2j_2k_2}\cdots t_{i_{m+l}j_{m+l}k_{m+l}}.
\label{eq:hoversion}
\end{equation}
From \eqref{eq:3rdoversion} and \eqref{eq:hoversion} it follows that 
 $\widehat{\mathcal T}^{(m+l)}$ is just a higher-order representation of ${\mathcal T}^{(m+l)}$.

A new tensor $\widehat{\mathcal T}^{(m+l)}_{\Lambda\Lambda S}$ is obtained from  $\widehat{\mathcal T}^{(m+l)}$
by applying two partial skew-symmetrizations and  one partial symmetrization as follows:
\begin{equation}
\widehat{\mathcal T}^{(m+l)}_{\Lambda\Lambda S} := 
\left[(\Lambda^m\outerprod \mathbf I_I\outerprod\dots\outerprod\mathbf I_I)\outerprod (\Lambda^m\outerprod \mathbf I_J\outerprod\dots\outerprod\mathbf I_J)\outerprod S^{m+l}\right]\widehat{\mathcal T}^{(m+l)}.
\label{eq:2skewsymmonesymm}
\end{equation}
 To obtain $\widehat{\mathcal T}^{(m+l)}_{\Lambda\Lambda S}$  we first skew-symmetrize    $\widehat{\mathcal T}^{(m+l)}$  with respect to the group of indices $\{1,\dots,m\}$ (the first $m$ ``$I$'' dimensions), then  we skew-symmetrize  the result  with respect to the group of indices $\{m+l+1,\dots,2m+l\}$ (the first $m$ ``$J$'' dimensions), and, finally, we symmetrize the result  with respect to the group of indices $\{2m+2l+1,\dots,3m+3l\}$ (all ``$K$'' dimensions).
 From  \eqref{symmetrization}, \eqref{skewsymmetrization}, and \eqref{eq:hoversion}, it follows that the 
$(i_1,\dots, i_{m+l}, j_1, \dots, j_{m+l}, k_1, \dots, k_{m+l})$th entry of the tensor $\widehat{\mathcal T}^{(m+l)}_{\Lambda\Lambda S}$
is equal to zero if some index is repeated in $i_1,\dots,i_m$ or $j_1,\dots,j_m$ and  is equal to
\begin{align*}
&\frac{1}{(m+l)!}
\sum\limits_{\substack{(s_1,\dots,s_{m+l})\in\\ P_{\{k_1,\dots,k_{m+l}\}}}} 
\Bigg[\frac{1}{m!}\sum\limits_{\substack{(q_1,\dots,q_m)\in \\ P_{\{j_1,\dots,j_m\}}}}
\sgn(q_1,\dots,q_m)\times\\
&\qquad\qquad\Bigg(\frac{1}{m!}\sum\limits_{\substack{(p_1,\dots,p_m)\in \\ P_{\{i_1,\dots,i_m\}}}}
\sgn(p_1,\dots,p_m)
\prod\limits_{u=1}^m t_{p_uq_us_u}\prod\limits_{v=1}^l t_{i_{m+v}j_{m+v}s_{m+v}}
\Bigg)\Bigg]=\\
&\frac{1}{(m+l)!}\sum\limits_{\substack{(s_1,\dots,s_{m+l})\in\\ P_{\{k_1,\dots,k_{m+l}\}}}} 
\Bigg[\frac{1}{m!}\sum\limits_{\substack{(q_1,\dots,q_m)\in \\ P_{\{j_1,\dots,j_m\}}}}
\sgn(q_1,\dots,q_m)\times\\
&\qquad\qquad\qquad\qquad\frac{1}{m!}\det\left[\begin{matrix}
   t_{i_1q_1s_1}&\dots & t_{i_1q_ms_m}\\
   \vdots       &\vdots& \vdots\\
   t_{i_mq_1s_1}&\dots & t_{i_mq_ms_m}
   \end{matrix}\right]
  \prod\limits_{v=1}^l t_{i_{m+v}j_{m+v}s_{m+v}}\Bigg]=\\
   &\frac{1}{m!(m+l)!}\sum\limits_{\substack{(s_1,\dots,s_{m+l})\in\\ P_{\{k_1,\dots,k_{m+l}\}}}}
      \det\left[\begin{matrix}
      t_{i_1j_1s_1}&\dots & t_{i_1j_ms_m}\\
      \vdots       &\vdots& \vdots\\
      t_{i_mj_1s_1}&\dots & t_{i_mj_ms_m}
      \end{matrix}\right]
      \prod\limits_{v=1}^l t_{i_{m+v}j_{m+v}s_{m+v}},
   \end{align*}
  otherwise (we used twice the  Leibniz formula for the determinant).
Thus, by \eqref{eq:29}, the tensor $\widehat{\mathcal T}^{(m+l)}_{\Lambda\Lambda S}$
 and the matrix $\mathbf R_{m,l}(\mathcal T)$ have the same entries (in  step 3 it will be shown that
 $\mathbf R_{m,l}(\mathcal T)$ is a matrix unfolding of $\widehat{\mathcal T}^{(m+l)}_{\Lambda\Lambda S}$).
 
Let us apply partial skew-symmetrizations and partial symmetrization to the right-hand side of \eqref{eq:2.6}:  from \eqref{eq:2.6}, \eqref{eq:2skewsymmonesymm}, (see also \eqref{eq:s2MLL}--\eqref{eq:l2MLL} for the properties of the outer product) it follows that
\begin{equation}
\begin{split}
\widehat{\mathcal T}^{(m+l)}_{\Lambda\Lambda S} =
\sum\limits_{r_1,\dots,r_{m+l}=1}^R
\big[\Lambda^m(\mathbf a_{r_1}\outerprod\dots\outerprod \mathbf a_{r_{m}})&\outerprod\mathbf a_{r_{m+1}}\outerprod\dots\outerprod\mathbf a_{r_{m+l}}\outerprod\\
\Lambda^m(\mathbf b_{r_1}\outerprod\dots\outerprod \mathbf b_{r_{m}})\outerprod\mathbf b_{r_{m+1}}\outerprod\dots\outerprod\mathbf b_{r_{m+l}}\big]\outerprod
  & S^{m+l}(\mathbf c_{r_1}\outerprod\dots\outerprod\mathbf c_{r_{m+l}})=\\
\sum\limits_{r_1,\dots,r_{m+l}=1}^R{\mathcal F}^{\mathbf A,\mathbf B}_{r_1,\dots,r_{m+l}}
 \outerprod   & S^{m+l}(\mathbf c_{r_1}\outerprod\dots\outerprod\mathbf c_{r_{m+l}}), 
\end{split}
\label{eq:2.7}
\end{equation}
where the expressions $\Lambda^m(\mathbf a_{r_1}\outerprod\dots\outerprod \mathbf a_{r_{m}})$ and
$\Lambda^m(\mathbf b_{r_1}\outerprod\dots\outerprod \mathbf b_{r_{m}})$ are defined in
\eqref{eq:jtheq1}, the expression $S^{m+l}(\mathbf c_{r_1}\outerprod\dots\outerprod\mathbf c_{r_{m+l}})$ is defined in
 \eqref{eq:jtheq0}, and, by definition,
 \begin{equation*}
 \begin{split}
 {\mathcal F}^{\mathbf A,\mathbf B}_{r_1,\dots,r_{m+l}}:=
  & \Lambda^m(\mathbf a_{r_1}\outerprod\dots\outerprod \mathbf a_{r_{m}})\outerprod\mathbf a_{r_{m+1}}\outerprod\dots\outerprod\mathbf a_{r_{m+l}}\outerprod\\
  & \Lambda^m(\mathbf b_{r_1}\outerprod\dots\outerprod \mathbf b_{r_{m}})\outerprod\mathbf b_{r_{m+1}}\outerprod\dots\outerprod\mathbf b_{r_{m+l}}
 \end{split}
  \end{equation*}
  (recall that the vectors $\mathbf a_r$ and $\mathbf b_r$ are columns of the matrices $\mathbf A$ and $\mathbf B$, respectively).

Note that by construction,  $S^{m+l}(\mathbf c_{r_1}\outerprod\dots\outerprod\mathbf c_{r_{m+l}})$ is a completely symmetric tensor ( that is, the expression $S^{m+l}(\mathbf c_{r_1}\outerprod\dots\outerprod\mathbf c_{r_{m+l}})$ does not change after any permutation of the vectors $\mathbf c_{r_1},\dots,\mathbf c_{r_{m+l}}$). Taking this fact into account we can group the summands in \eqref{eq:2.7} as follows
\begin{equation}
\widehat{\mathcal T}^{(m+l)}_{\Lambda\Lambda S} =
\sum\limits_{1\leq r_1\leq\dots\leq r_{m+l}\leq R}\bigg(
 \sum\limits_{\substack{(p_1,\dots,p_{m+l})\in\\ P_{\{r_1,\dots,r_{m+l}\}}}}
 {\mathcal F}^{\mathbf A,\mathbf B}_{p_1,\dots,p_{m+l}}\bigg)
\outerprod
 S^{m+l}(\mathbf c_{r_1}\outerprod\dots\outerprod\mathbf c_{r_{m+l}}). 
\label{eq:2.8}
\end{equation}
\subsection{Step 3: Reshaping (unfolding) of $\widehat{\mathcal T}^{(m+l)}_{\Lambda\Lambda S}$ into the matrix $\mathbf R_{m,l}(\mathcal T)$}
We define the matricization operation 
$$
 \textup{Matr}:\ \mathbb R^{I\times\dots\times I\times J\times\dots\times J\times K\times\dots\times K}\ \rightarrow\  
 \mathbb R^{I^{m+l}J^{m+l}\times K^{m+l}}
$$
as follows: the $(i_1,\dots,i_{m+l},j_1,\dots,j_{m+l},k_1,\dots,k_{m+l})$th entry of a tensor is mapped to the $((\tilde i -1)J^{m+l}+\tilde j,\tilde k )$th entry of  a matrix, where $\tilde i$, $\tilde j$, and $\tilde k$ are defined in
\eqref{eq:tildei}, \eqref{eq:tildej}, and \eqref{eq:tildek}, respectively. One can easily verify that 
\begin{equation}
\begin{split}
\textup{Matr}&\left(
\mathbf a_{i_1}\outerprod\dots\outerprod\mathbf a_{i_{m+l}}\outerprod
\mathbf b_{j_1}\outerprod\dots\outerprod\mathbf b_{j_{m+l}}\outerprod
\mathbf c_{k_1}\outerprod\dots\outerprod\mathbf c_{k_{m+l}}
\right)=\\
&\left[
\mathbf a_{i_1}\otimes\dots\otimes\mathbf a_{i_{m+l}}\otimes
\mathbf b_{j_1}\otimes\dots\otimes\mathbf b_{j_{m+l}}\right]
(\mathbf c_{k_1}\otimes\dots\otimes\mathbf c_{k_{m+l}})^T
\end{split}
\label{eq:matrrank1}
\end{equation}
and that $\mathbf R_{m,l}(\mathcal T)=\textup{Matr}(\widehat{\mathcal T}^{(m+l)}_{\Lambda\Lambda S})$.

What is left to show is that the matricization of the right-hand side of \eqref{eq:2.8} coincides with the matrix ${\mathbf\Phi}_{m,l}(\mathbf A,\mathbf B)\mathbf S_{m+l}(\mathbf C)^T$.
In the sequel, when no confusion is possible, we will use $S^k$ and $\Lambda^k$ to denote ``symmetrization'' and ``skew-symmetrization''
of vector representations of a certain tensor: if $\mathbf d_1,\dots,\mathbf d_k\in\mathbb R^L$, then the vectors $S^k(\mathbf d_1\otimes\dots\otimes\mathbf d_k)$ and $\Lambda^k(\mathbf d_1\otimes\dots\otimes\mathbf d_k)$ are computed in the same way as in 
\eqref{eq:jtheq0}--\eqref{eq:jtheq1} but with ``$\outerprod$'' replaced by ``$\otimes$'':
\begin{align}
 S^k(\mathbf d_1\otimes\dots\otimes \mathbf d_k) &= 
\frac{1}{k!}\sum\limits_{(p_1,\dots,p_k)\in P_{\{1,\dots,k\}}}  \mathbf d_{p_1}\otimes\dots\otimes \mathbf d_{p_k},\nonumber\\
 \Lambda^k(\mathbf d_1\otimes\dots\otimes \mathbf d_k) &= 
\frac{1}{k!}\sum\limits_{(p_1,\dots,p_k)\in P_{\{1,\dots,k\}}}  \sigma(p_1,\dots,p_k)\mathbf d_{p_1}\otimes\dots\otimes \mathbf d_{p_k}.
\label{eq:jtheq1kron}
\end{align}
Hence, by \eqref{eq:2.7}, \eqref{eq:2.8}, and \eqref{eq:matrrank1}
\begin{align}
&\mathbf R_{m,l}(\mathcal T)=\textup{Matr}(\widehat{\mathcal T}^{(m+l)}_{\Lambda\Lambda S})=\notag\\
&\sum\limits_{1\leq r_1\leq\dots\leq r_{m+l}\leq R}\bigg(
 \sum\limits_{\substack{(s_1,\dots,s_{m+l})\in\\ P_{\{r_1,\dots,r_{m+l}\}}}}
 {\mathbf f}^{\mathbf A,\mathbf B}_{s_1,\dots,s_{m+l}}\bigg)
 S^{m+l}(\mathbf c_{r_1}\otimes\dots\otimes\mathbf c_{r_{m+l}})^T= \label{eq:3.17}\\
& \sum\limits_{1\leq r_1\leq\dots\leq r_{m+l}\leq R}{\mathbf{\phi}}(\mathbf A,\mathbf B)_{r_1,\dots,r_{m+l}}
  S^{m+l}(\mathbf c_{r_1}\otimes\dots\otimes\mathbf c_{r_{m+l}})^T,\notag
 \end{align}
 where
\begin{equation*}
\begin{split}
  {\mathbf f}^{\mathbf A,\mathbf B}_{s_1,\dots,s_{m+l}}:=
   &\Lambda^m(\mathbf a_{s_1}\otimes\dots\otimes \mathbf a_{s_{m}})\otimes\mathbf a_{s_{m+1}}\otimes\dots\otimes\mathbf a_{s_{m+l}}\otimes\\
  &\Lambda^m(\mathbf b_{s_1}\otimes\dots\otimes \mathbf b_{s_{m}})\otimes\mathbf b_{s_{m+1}}\otimes\dots\otimes\mathbf b_{s_{m+l}},\\
  {\mathbf{\phi}}(\mathbf A,\mathbf B)_{r_1,\dots,r_{m+l}}:=&\sum\limits_{\substack{(s_1,\dots,s_{m+l})\in\\ P_{\{r_1,\dots,r_{m+l}\}}}}
  {\mathbf f}^{\mathbf A,\mathbf B}_{s_1,\dots,s_{m+l}}.
  \end{split}
  \end{equation*}
  We show that ${\mathbf{\phi}}(\mathbf A,\mathbf B)_{r_1,\dots,r_{m+l}}$ is the zero vector if
  the set $\{r_1,\dots,r_{m+l}\}$  has fewer than $m$  distinct elements and that
  ${\mathbf{\phi}}(\mathbf A,\mathbf B)_{r_1,\dots,r_{m+l}}$ is a column of
  ${\mathbf\Phi}_{m,l}(\mathbf A,\mathbf B)$ otherwise.
  From \eqref{eq:jtheq1kron} and the  Leibniz formula for the determinant it follows that
  the entries of the vector  $\Lambda^k(\mathbf d_1\otimes\dots\otimes \mathbf d_k)$ are all possible $k\times k$ minors of the
  matrix $\mathbf D:=[\mathbf d_1\ \dots\  \mathbf d_k]$ divided by $k!$. In particular, if some of the vectors $\mathbf d_i$ coincide, then 
  $\Lambda^k(\mathbf d_1\otimes\dots\otimes \mathbf d_k)$ is the zero vector.
  Hence, the vector ${\mathbf f}^{\mathbf A,\mathbf B}_{s_1,\dots,s_{m+l}}$ has entries
   \begin{equation*}
     \begin{split}
     \frac{1}{(m!)^2}
       \det&\left[\begin{matrix}
       a_{i_1s_1}&\dots & a_{i_1s_m}\\
       \vdots       &\vdots& \vdots\\
       a_{i_ms_1}&\dots & a_{i_ms_m}
       \end{matrix}\right]\cdot
       \det\left[\begin{matrix}
                b_{j_1s_1}&\dots & b_{j_1s_m}\\
                \vdots       &\vdots& \vdots\\
                b_{j_ms_1}&\dots & b_{j_ms_m}
                \end{matrix}\right]\cdot\\
      &\qquad  a_{i_{m+1}s_{m+1}}\cdots a_{i_{m+l}s_{m+l}}\cdot
      b_{j_{m+1}s_{m+1}}\cdots b_{j_{m+l}s_{m+l}},
       \end{split} 
       \end{equation*}
       where
$i_1,\dots,i_{m+l}\in\{1,\dots,I\}$ and $j_1,\dots,j_{m+l}\in\{1,\dots,J\}$.
   In particular, if the set $\{r_1,\dots,r_{m+l}\}$  has fewer than $m$  distinct elements, then
  ${\mathbf f}^{\mathbf A,\mathbf B}_{s_1,\dots,s_{m+l}}$ are zero vectors for all
  $(s_1,\dots,s_{m+l})\in P_{\{r_1,\dots,r_{m+l}\}}$, yielding that ${\mathbf{\phi}}(\mathbf A,\mathbf B)_{r_1,\dots,r_{m+l}}$ is the zero vector.
   Hence, by Definition \ref{def:matrixPhi}, the matrix
  ${\mathbf\Phi}_{m,l}(\mathbf A,\mathbf B)$ has columns ${\mathbf{\phi}}(\mathbf A,\mathbf B)_{r_1,\dots,r_{m+l}}$, where
  $(r_1,\dots,r_{m+l})$ satisfies \eqref{eq:mplusltuples}.
  Thus, \eqref{eq:3.17} coincides with \eqref{eq:mainidentity}.
\section{Proof of Lemma \ref{lemma:suppl}} \label{section:Appendix}
The proof of Lemma \ref{lemma:suppl} is based on the
following simple generalization of the rank-nullity theorem and relies on two bounds that will be obtained in 
in \ref{App2} and  \ref{App1}, respectively.
\begin{lemma}\label{lemmasupplmats12}
Let $\mathbf X$ be a matrix and $ E$ be a subspace such that $\textup{range}(\mathbf X^T)\subseteq E$. Then
\begin{itemize}
\item[\textup{(i)}] $\dim(\ker (\mathbf X)\cap E) + \dim(\mathbf X(E)) = \dim E$;
\item[\textup{(ii)}] $\textup{range}(\mathbf X)=\mathbf X(E)$,
\end{itemize}
where the subspace $\mathbf X(E)$ denotes the image of $E$ under $\mathbf X$.
\end{lemma}
\begin{proof}
Let $\mathbf P$ be a matrix whose columns form a basis for the subspace $E$. Then
$\dim(\ker (\mathbf X)\cap E)=\dim(\ker (\mathbf X\mathbf P))$,  $\mathbf X(E) = \textup{range}(\mathbf X\mathbf P)$, and
the matrix $\mathbf X\mathbf P$ has $\dim E$ columns.
Hence, by the rank-nullity theorem,
$$
\dim(\ker (\mathbf X)\cap E) + \dim(\mathbf X(E)) = 
\dim(\ker (\mathbf X\mathbf P))+\dim(\mathbf X\mathbf P)=
\dim E. 
$$
Since, $\textup{range}(\mathbf X^T)\subseteq \textup{range}(\mathbf P)$, it follows that $\textup{range}(\mathbf X)=
\textup{range}(\mathbf X\mathbf P)=\mathbf X(E)$.\qquad
\end{proof}
{\em Proof of Lemma \ref{lemma:suppl}.}
  We set $\mathbf X=\mathbf S_{m+l}(\mathbf C)^T$ and $E=S^{m+l}(\mathbb R^{K^{m+l}})$. Then statement \textup{(iii)} follows from Lemma \ref{lemmasupplmats12} \textup{(ii)}.
By Lemma \ref{lemmasupplmats12} \textup{(i)}, 
\begin{equation}
\label{eq:lastequation}
\dim  \left(\ker (\mathbf S_{m+l}(\mathbf C)^T)\bigcap S^{m+l}(\mathbb R^{K^{m+l}})\right)+
r_{\mathbf S_{m+l}(\mathbf C)^T}= C^{m+l}_{K+m+l-1}.
\end{equation}
In \ref{App2} and  \ref{App1} we prove that the summands in \eqref{eq:lastequation} are bounded as
\begin{align*}
\dim  \left(\ker (\mathbf S_{m+l}(\mathbf C)^T)\bigcap S^{m+l}(\mathbb R^{K^{m+l}})\right)&\geq C^{K-1}_R,
\text{ and}
\\
r_{\mathbf S_{m+l}(\mathbf C)^T}&\geq C^{m+l}_{R+l+1}-C^{m-1}_R, 
\end{align*}
respectively.
Since $C^{K-1}_R + (C^{m+l}_{R+l+1}-C^{m-1}_R) = C^{m-1}_R + (C^{m+l}_{K+m+l-1}-C^{m-1}_R)=C^{m+l}_{K+m+l-1}$,
statements \textup{(i)} and \textup{(iv)} follow from \eqref{eq:lastequation}. Statement \textup{(ii)} follows from
statement \textup{(i)} and Lemma \ref{lemma5.1111} below.
\qed
\subsection{A lower bound on 
$\dim  \left(\ker (\mathbf S_{m+l}(\mathbf C)^T)\bigcap S^{m+l}(\mathbb R^{K^{m+l}})\right)$}\label{App2}
In this subsection we prove the following result.
\begin{lemma}\label{lemma5.1111}
Let $\mathbf C\in\mathbb R^{K\times R}$, $k_{\mathbf C}=K$, $m=R-K+2$, $l\geq 0$,  let $\mathbf F$  satisfy \eqref{P1}--\eqref{P2}, and let $\mathbf F^{(m+l)}$ be defined by \eqref{eq:Fm}. Then
\begin{itemize}
\item[\textup{(i)}] The matrix $\mathbf F^{(m+l)}$ has full column rank, that is $r_{\mathbf F^{(m+l)}}=C^{K-1}_R$;
\item[\textup{(ii)}] $
\ker (\mathbf S_{m+l}(\mathbf C)^T)\bigcap S^{m+l}(\mathbb R^{K^{m+l}})\supset\textup{range}\left(\mathbf F^{(m+l)}\right)$.
\end{itemize}
In particular, 
 $\dim  \left(\ker (\mathbf S_{m+l}(\mathbf C)^T)\bigcap S^{m+l}(\mathbb R^{K^{m+l}})\right)\geq C^{K-1}_R$.
\end{lemma}

\begin{proof}
Statement \textup{(i)} was proved in  \cite[Proposition 1.10]{LinkGEVD}. 

Let  $\mathbf f$ be a column of the matrix $\mathbf F$. Then $\mathbf f^{(m+l)}$ is a column of $\mathbf F^{(m+l)}$.
It is clear that $\mathbf f^{(m+l)}\in S^{m+l}(\mathbb R^{K^{m+l}})$. To prove \textup{(ii)} we need to show that
$\mathbf S_{m+l}(\mathbf C)^T\mathbf f^{(m+l)}=\mathbf 0$. By Definition \ref{def:matrixS}, the $(r_1,\dots,r_{m+l})$th entry of the vector
$\mathbf S_{m+l}(\mathbf C)^T\mathbf f^{(m+l)}$ is 
 \begin{equation*}
\begin{split}
          &\Bigg(\frac{1}{(m+l)!} \sum\limits_{\substack{(s_1,\dots,s_{m+l})\in\\ P_{\{r_1,\dots,r_{m+l}\}}}}\mathbf c_{s_1}\otimes\dots\otimes\mathbf c_{s_{m+l}}\Bigg)^T\mathbf f^{(m+l)}=\\
          &\frac{1}{(m+l)!} \sum\limits_{\substack{(s_1,\dots,s_{m+l})\in\\ P_{\{r_1,\dots,r_{m+l}\}}}}
           (\mathbf c_{s_1}^T\mathbf f)\cdots(\mathbf c_{s_{m+l}}^T\mathbf f)=
          (\mathbf c_{r_1}^T\mathbf f)\cdots(\mathbf c_{r_{m+l}}^T\mathbf f).
          \end{split}
\end{equation*}
Since the vector $\mathbf f$ is orthogonal to exactly $K-1$ columns of $\mathbf C$, the fact that at least
 $m$ indices of $r_1,\dots,r_{m+l}$ are distinct, and $m+l\geq R-K+2+l>R-(K-1)$ it follows that  $(\mathbf c_{r_1}^T\mathbf f)\cdots(\mathbf c_{r_{m+l}}^T\mathbf f)=0$, which completes the proof of \textup{(ii)}.
\end{proof}
\subsection{A lower bound on the rank of the matrix $\mathbf S_{m+l}(\mathbf C)$}\label{App1}
We need some additional notation.
Let $m\geq 2$, $l\geq 0$, $p\geq 0$ and 
$m\leq m+l-p\leq R$.
 With an $(m+l-p)$-tuple
$$
(i_1,\dots,i_{m+l-p})\text{ such that } 1\leq i_1<\dots <i_{m+l-p}\leq R
$$
we associate  the set 
$$
E_{i_1,\dots,i_{m+l-p}}:=\{(i_1,\dots,i_{m+l-p},i_{q_1},\dots,i_{q_p}):\ 1\leq q_1\leq\dots\leq q_p\leq 2+l-p\}.
$$
In other words, the set $E_{i_1,\dots,i_{m+l-p}}$ consists of the $(m+l)$-tuples that are obtained by merging the $(m+l-p)$-tuple
 $(i_1,\dots,i_{m+l-p})$ with  $p$-combinations with repetitions of the set $\{i_1,\dots,i_{2+l-p}\}$.
It is clear that for a fixed $p$ there exist $C^{m+l-p}_R$ sets $E_{i_1,\dots,i_{m+l-p}}$ and each set
 $E_{i_1,\dots,i_{m+l-p}}$ contains $C^{p}_{2+l-p+p-1}=C^p_{l+1}$ $(m+l)$-tuples.
Let $E$  be the union of all sets $E_{i_1,\dots,i_{m+l-p}}$.
Then  the set  $E$ contains exactly
$$
C^0_{l+1}C^{m+l}_R+C^1_{l+1}C^{m+l-1}_R+\dots+C^l_{l+1}C^{m}_R = C^{m+l}_{R+l+1}-C^{m-1}_R
$$
$(m+l)$-tuples (we follow the convention that $C^{m+l-p}_R:=0$ if $m+l-p>R$).
Since, by construction,  each $(m+l)$-tuple of $E_{i_1,\dots,i_{m+l-p}}$ contains exactly 
$m+l-p\geq m$ distinct elements, it follows that   each $(m+l)$-tuple of $E$  contains
at least $m$ distinct elements. 
Let $\mathbf S^E_{m+l}(\mathbf C)$ denote the  $K^m\times (C^{m+l}_{R+l+1}-C^{m-1}_R)$ matrix with columns \eqref{eq:def2.3star}, where  $(r_1,\dots,r_{m+l})\in E$.
Then  $\mathbf S^E_{m+l}(\mathbf C)$ is a submatrix of 
$\mathbf S_{m+l}(\mathbf C)$.  We have the following lemma.
\begin{lemma}\label{lemma5.1}
Let $\mathbf C\in\mathbb R^{K\times R}$, $k_{\mathbf C}=K$, and $m=R-K+2$. Then the matrix
$\mathbf S^E_{m+l}(\mathbf C)$ has full column rank. In particular,  $r_{\mathbf S_{m+l}(\mathbf C)} \geq C^{m+l}_{R+l+1}-C^{m-1}_R$.
\end{lemma}
\begin{proof}
Suppose that there exists $\mathbf f\in\mathbb R^{C^{m+l}_{R+l+1}-C^{m-1}_R}$ such that $\mathbf S^E_{m+l}(\mathbf C){\mathbf f}=\mathbf 0$.
We show that ${\mathbf f}=\mathbf 0$. We assume that the entries of $\mathbf f$ are indexed by  
$(m+l)$-tuples $(r_1,\dots,r_{m+l})\in E$, that is,
in $\mathbf S^E_{m+l}(\mathbf C){\mathbf f}=\mathbf 0$ the column of $\mathbf S^E_{m+l}(\mathbf C)$
associated with the $(m+l)$-tuple $(i_1,\dots,i_{m+l-p},i_{q_1},\dots,i_{q_p})$ is multiplied by
$f_{i_1,\dots,i_{m+l-p},i_{q_1},\dots,i_{q_p}}$. 

To show that all entries $f_{i_1,\dots,i_{m+l-p},i_{q_1},\dots,i_{q_p}}$ are  zero we proceed by induction on $p=l,l-1,\dots,\max(0,m+l-R)$: in the $p$th step we assume that the identities 
\begin{multline*}
 f_{i_1,\dots,i_{m+l-\tilde p},i_{q_1},\dots,i_{q_{\tilde p}}}=0,\ \ \text{where}\\
1\leq i_1<\dots <i_{m+l-\tilde p}\leq R, \ \  1\leq q_1\leq\dots\leq q_{\tilde p}\leq 2+l-\tilde p
\end{multline*}
hold for $\tilde p =l,l-1,\dots,p-1$ and prove that the identities hold for $\tilde p=p$.

\textup{(i)} Induction hypothesis: $p=l$. We show that  
$$
f_{i_1,\dots,i_m,i_{q_1},\dots,i_{q_l}}=0\quad  \text{ for } \quad
1\leq i_1<\dots<i_{m}\leq R, \quad 1\leq q_1\leq\dots\leq q_l\leq 2.
$$
We give the proof for the case $i_1=1,\dots,i_m=m$, the other cases follow similarly.
Thus, we show that
$
f_{1,\dots,m,q_1,\dots,q_l}=0$ for $1\leq q_1\leq\dots\leq q_l\leq 2$.

Since $k_{\mathbf C}=K$, the square matrix $\widetilde{\mathbf C}:=[\mathbf c_1\ \mathbf c_2\ \mathbf c_{m+1}\dots\ \mathbf c_R]$ is nonsingular.
Let $\mathbf u_1$ and $\mathbf u_2$ denote the first and the second column of 
$\widetilde{\mathbf C}^{-T}$, respectively. Then
\begin{equation}
\label{eq:u1u2}
\left[
\begin{matrix}
\mathbf u_1^T\\
\mathbf u_2^T
\end{matrix}
\right]\widetilde{\mathbf C}=
\left[
\begin{matrix}
 1&0&0&\dots&0\\
 0&1&0&\dots&0\\
\end{matrix}
\right].
\end{equation}
Let $\mathbf x=t_1\mathbf u_1+t_2\mathbf u_2$. Then 
the vector $\mathbf x^{(m+l)}:=\mathbf x\otimes\dots\otimes\mathbf x$ is orthogonal to the columns of the matrix $\mathbf S^E_{m+l}(\mathbf C)$ indexed by
the $(m+l)$-tuples $(r_1,\dots,r_{m+l})\in E\setminus E_{1,\dots, m}$.
Indeed, if $\{r_1,\dots,r_{m+l}\}\setminus \{1,\dots,m\}\ne\emptyset$, then  by \eqref{eq:def2.3star} and \eqref{eq:u1u2},
  \begin{equation}
     \frac{\mathbf x^{(m+l)T}}{(m+l)!} \sum\limits_{\substack{(s_1,\dots,s_{m+l})\in\\ P_{\{r_1,\dots,r_{m+l}\}}}}\mathbf c_{s_1}\otimes\dots\otimes\mathbf c_{s_{m+l}}=(\mathbf x^T \mathbf c_{r_1})\cdots (\mathbf x^T \mathbf c_{r_{m+l}})=0.
       \label{eq:def2.35.5}
      \end{equation}
Hence
\begin{equation}
\label{eq:last}
\begin{split}
0&= \mathbf x^{(m+l)T}\mathbf S^E_{m+l}(\mathbf C){\mathbf f}=
\sum\limits_{(r_1,\dots,r_{m+l})\in E_{1,\dots, m}}
(\mathbf x^T \mathbf c_{r_1})\cdots (\mathbf x^T \mathbf c_{r_{m+l}})f_{r_1,\dots,r_{m+l}}=\\
&\sum\limits_{1\leq q_1\leq\dots\leq q_l\leq 2}
(\mathbf x^T \mathbf c_1)\cdots (\mathbf x^T \mathbf c_m)(\mathbf x^T \mathbf c_{q_1})\cdots (\mathbf x^T \mathbf c_{q_l}) f_{1,\dots, m, q_1,\dots, q_l}=\\
&(\mathbf x^T \mathbf c_1)\cdots (\mathbf x^T \mathbf c_m)\sum\limits_{1\leq q_1\leq\dots\leq q_l\leq 2}
(\mathbf x^T \mathbf c_{q_1})\cdots (\mathbf x^T \mathbf c_{q_l}) f_{1,\dots, m, q_1,\dots, q_l}.
\end{split}
\end{equation}
Since $k_{\mathbf C}=K$, at most one of the vectors $\mathbf u_1$ and $\mathbf u_2$ can be orthogonal to any of the vectors
$\mathbf c_3,\dots,\mathbf c_m$. Hence, 
$$
(\mathbf x^T \mathbf c_1)\cdots (\mathbf x^T \mathbf c_m)=
t_1t_2(t_1\mathbf u_1^T \mathbf c_3+t_2\mathbf u_2^T \mathbf c_3)\cdots (t_1\mathbf u_1^T \mathbf c_m+t_2\mathbf u_2^T \mathbf c_m)\ne 0
$$
for generic $t_1,t_2\in \mathbb R$. Hence, by \eqref{eq:last},
\begin{equation}
\label{eq:pol=0}
\sum\limits_{1\leq q_1\leq\dots\leq q_l\leq 2}
(\mathbf x^T \mathbf c_{q_1})\cdots (\mathbf x^T \mathbf c_{q_l}) f_{1,\dots, m, q_1,\dots, q_l}=0
\end{equation}
for generic $t_1,t_2\in \mathbb R$. 
By construction of $\mathbf x$, the $l+1$ products\\ $(\mathbf x^T \mathbf c_{q_1})\cdots (\mathbf x^T \mathbf c_{q_l})$, $1\leq q_1\leq\dots\leq q_l\leq 2$,
coincide with the monomials $t_1^lt_2^0,t_1^{l-1}t_2^1,\dots,t_1^0t_2^l$. Thus, identity 
\eqref{eq:pol=0} expresses the fact that a polynomial in $t_1$ and $t_2$ 
with coefficients $f_{1,\dots, m, q_1,\dots, q_l}$ vanishes for generic  $t_1,t_2\in \mathbb R$.
It is well known that this is possible only if the polynomial is  identically zero, yielding that $f_{1,\dots,m,q_1,\dots,q_l}=0$ for
$1\leq q_1\leq\dots\leq q_l\leq 2$.

\textup{(ii)} Inductive step. 
We show that
\begin{multline*}
 f_{i_1,\dots,i_{m+l-p},i_{q_1},\dots,i_{q_{p}}}=0\ \text{ for }\\
1\leq i_1<\dots <i_{m+l- p}\leq R, \ \  1\leq q_1\leq\dots\leq q_{ p}\leq 2+l- p
\end{multline*}
or, equivalently, that $f_{i_1,\dots,i_{m+l-p},i_{q_1},\dots,i_{q_p}}=0$  for 
$$
(i_1,\dots,i_{m+l-p},i_{q_1},\dots,i_{q_p})\in \bigcup\limits_{1\leq i_1<\dots<i_{m+l-p}\leq R} E_{i_1,\dots,i_{m+l-p}}.
$$
We give the proof for the case $i_1=1,\dots,i_{m+l-p}=m+l-p$, the other cases follow similarly.
Thus, we show that
$
f_{1,\dots,{m+l-p},q_1,\dots,q_p}=0$ for $1\leq q_1\leq\dots\leq q_p\leq {2+l-p}$.
The derivation is very similar to that of the induction hypothesis.

Since $k_{\mathbf C}=K$, the $K\times K$ matrix $\widetilde{\mathbf C}:=[\mathbf c_1\ \dots \mathbf c_{2+l-p}\ \mathbf c_{m+l-p+1}\dots\ \mathbf c_R]$ is nonsingular.
Let $\mathbf u_1,\dots,\mathbf u_{2+l-p}$ denote the first $2+l-p$ columns of 
$\widetilde{\mathbf C}^{-T}$. Then
\begin{equation}
\label{eq:u1u2ii}
\left[
\begin{matrix}
\mathbf u_1^T\\
\vdots\\
\mathbf u_{2+l-p}^T
\end{matrix}
\right] 
\widetilde{\mathbf C}=
\left[
\begin{matrix}
 1&0& \dots& 0&0&\dots&0\\
 0&1& \dots& 0&0&\dots&0\\
 \vdots&\vdots  &\vdots &\vdots  &\vdots &\vdots&\vdots\\
 0&0& \dots& 1&0&\dots&0
\end{matrix}
\right].
\end{equation}
Let $\mathbf x=t_1\mathbf u_1+\dots +t_{2+l-p}\mathbf u_{2+l-p}$.  
Let
$$
E_{\tilde p<p} := \bigcup_{\tilde p<p}\ \bigcup\limits_{1\leq i_1<\dots <i_{m+l-\tilde p}\leq R} E_{i_1,\dots,i_{m+l-\tilde p}}
$$
and let the sets $E_{\tilde p>p}$ and $E_{\tilde p=p}$ be defined similarly. Then 
$E=E_{\tilde p<p}\cup E_{\tilde p>p}\cup E_{\tilde p=p}$.
 Then, by  \eqref{eq:def2.3star}, \eqref{eq:def2.35.5}, and \eqref{eq:u1u2ii}, 
the vector $\mathbf x^{(m+l)}:=\mathbf x\otimes\dots\otimes\mathbf x$ is orthogonal to  the columns of the matrix $\mathbf S^E_{m+l}(\mathbf C)$ indexed
 by the 
$(m+l)$-tuples 
$$
(r_1,\dots,r_{m+l})\in E_{\tilde p<p} \cup \left(E_{\tilde p=p}\setminus E_{1,\dots,m+l-p}\right).
$$
Hence, similarly to \eqref{eq:last} we obtain
\begin{equation*}
\begin{split}
0&= \mathbf x^{(m+l)T}\mathbf S^E_{m+l}(\mathbf C){\mathbf f}=
\sum\limits_{(r_1,\dots,r_{m+l})\in E}
(\mathbf x^T \mathbf c_{r_1})\cdots (\mathbf x^T \mathbf c_{r_{m+l}})f_{r_1,\dots,r_{m+l}}=\\
& \sum\limits_{(r_1,\dots,r_{m+l})\in E_{\tilde p>p}\cup E_{1,\dots,m+l-p}}
(\mathbf x^T \mathbf c_{r_1})\cdots (\mathbf x^T \mathbf c_{r_{m+l}})f_{r_1,\dots,r_{m+l}}.
\end{split}
\end{equation*}
Since, by the induction assumption, $f_{r_1,\dots,r_{m+l}}=0$ for $(r_1,\dots,r_{m+l})\in E_{\tilde p>p}$, we have
\begin{equation}
\label{eq:lastii}
\begin{split}
0&=\sum\limits_{(r_1,\dots,r_{m+l})\in  E_{1,\dots,m+l-p}}
(\mathbf x^T \mathbf c_{r_1})\cdots (\mathbf x^T \mathbf c_{r_{m+l}})f_{r_1,\dots,r_{m+l}}=\\
&\sum\limits_{ 1\leq q_1\leq\dots\leq q_{ p}\leq 2+l- p}
 (\mathbf x^T \mathbf c_1)\cdots (\mathbf x^T \mathbf c_{m+l-p})
 (\mathbf x^T \mathbf c_{q_1})\cdots (\mathbf x^T \mathbf c_{q_p})
 f_{1,\dots,{m+l-p},q_1,\dots,q_p}=\\
 & (\mathbf x^T \mathbf c_1)\cdots (\mathbf x^T \mathbf c_{m+l-p})\sum\limits_{ 1\leq q_1\leq\dots\leq q_{ p}\leq 2+l- p}
    (\mathbf x^T \mathbf c_{q_1})\cdots (\mathbf x^T \mathbf c_{q_p})
  f_{1,\dots,{m+l-p},q_1,\dots,q_p}.
\end{split}
\end{equation}
Since $k_{\mathbf C}=K$,  at most $1+l-p$ of the vectors $\mathbf u_1,\dots,\mathbf u_{2+l-p}$ can be  orthogonal to any of the vectors
$\mathbf c_{3+l-p},\dots,\mathbf c_{m+l-p}$. Hence, 
\begin{equation}
\begin{split}
&(\mathbf x^T \mathbf c_1)\cdots (\mathbf x^T \mathbf c_{m+l-p})=\\
& t_1\cdots t_{2+l-p}(t_1\mathbf u_1^T \mathbf c_{3+l-p}+\dots+t_{2+l-p}\mathbf u_{2+l-p}^T \mathbf c_{3+l-p})\cdots\\
& (t_1\mathbf u_1^T \mathbf c_{m+l-p}+\dots+t_{2+l-p}\mathbf u_{2+l-p}^T \mathbf c_{m+l-p})\ne 0
\end{split}
\end{equation}
for generic $t_1,\dots,t_{2+l-p}\in \mathbb R$. Hence, by \eqref{eq:lastii},
\begin{equation}
\label{eq:pol=0ii}
\sum\limits_{ 1\leq q_1\leq\dots\leq q_{ p}\leq 2+l- p}
    (\mathbf x^T \mathbf c_{q_1})\cdots (\mathbf x^T \mathbf c_{q_p})
  f_{1,\dots,{m+l-p},q_1,\dots,q_p}
=0
\end{equation}
for generic $t_1,\dots,t_{2+l-p}\in \mathbb R$. 
By construction of $\mathbf x$, the $C^p_{l+1}$ products $(\mathbf x^T \mathbf c_{q_1})\cdots (\mathbf x^T \mathbf c_{q_p})$,
$1\leq q_1\leq\dots\leq q_p\leq 2+l-p$,
coincide with the  monomials $\{t_1^{\alpha_1}\cdots t_{2+l-p}^{\alpha_{2+l-p}}\}_{\alpha_1+\dots+\alpha_{2+l-p}=p}$. Thus,  identity 
\eqref{eq:pol=0ii} expresses the fact that a polynomial in  $t_1,\dots,t_{2+l-p}$ 
with coefficients $f_{1,\dots,{m+l-p},q_1,\dots,q_p}$ vanishes for generic  $t_1,\dots,t_{2+l-p}\in \mathbb R$.
It is well known that this is possible only if the polynomial is  identically zero, yielding that $f_{1,\dots,{m+l-p},q_1,\dots,q_p}=0$ for $1\leq q_1\leq\dots\leq q_p\leq 2+l-p$.
\end{proof}





\bibliographystyle{plain}

\begin{thebibliography}{10}
 
 \bibitem{Bocci2013}
 C.~Bocci, L.~Chiantini, and G.~Ottaviani.
 \newblock Refined methods for the identifiability of tensors.
 \newblock {\em Ann. Mat. Pur. Appl.}, pages 1--12, 2013.
 
 \bibitem{1970_Carroll_Chang}
 J.~Carroll and J.-J. Chang.
 \newblock {A}nalysis of individual differences in multidimensional scaling via
   an {N}-way generalization of ``{E}ckart-{Y}oung'' decomposition.
 \newblock {\em Psychometrika}, 35:283--319, 1970.
 
 \bibitem{ChiantiniandOttaviani}
 L.~Chiantini and G.~Ottaviani.
 \newblock On generic identifiability of 3-tensors of small rank.
 \newblock {\em SIAM J. Matrix Anal. Appl.}, 33(3):1018--1037, 2012.
 
 \bibitem{LievenCichocki2013}
 A.~Cichocki, D.~Mandic, C.~Caiafa, A-H. Phan, G.~Zhou, Q.~Zhao, and
   L.~De~Lathauwer.
 \newblock Tensor decompositions for signal processing applications.{ F}rom
   two-way to multiway component analysis.
 \newblock {\em IEEE Signal Process. Mag.}, 32:145--163, March 2015.
 
 \bibitem{ComoJ10}
 P.~Comon and C.~Jutten, editors.
 \newblock {\em {H}andbook of {B}lind {S}ource {S}eparation, {I}ndependent
   {C}omponent {A}nalysis and {A}pplications}.
 \newblock Academic Press, Oxford, UK, 2010.
 
 \bibitem{DeLathauwer2006}
 L.~De~Lathauwer.
 \newblock A link between the canonical decomposition in multilinear algebra and
   simultaneous matrix diagonalization.
 \newblock {\em SIAM J. Matrix Anal. Appl.}, 28:642--666, 2006.
 
 \bibitem{PartI}
 I.~Domanov and L.~De~Lathauwer.
 \newblock {O}n the uniqueness of the canonical polyadic decomposition of
   third-order tensors --- {P}art {I}: {B}asic results and uniqueness of one
   factor matrix.
 \newblock {\em SIAM J. Matrix Anal. Appl.}, 34:855--875, 2013.
 
 \bibitem{PartII}
 I.~Domanov and L.~De~Lathauwer.
 \newblock {O}n the uniqueness of the canonical polyadic decomposition of
   third-order tensors --- {P}art {II}: {O}verall uniqueness.
 \newblock {\em SIAM J. Matrix Anal. Appl.}, 34:876--903, 2013.
 
 \bibitem{LinkGEVD}
 I.~Domanov and L.~De~Lathauwer.
 \newblock Canonical polyadic decomposition of third-order tensors: reduction to
   generalized eigenvalue decomposition.
 \newblock {\em SIAM J. Matrix Anal. Appl.}, 35(2):636--660, 2014.
 
 \bibitem{AlgGeom1}
 I.~Domanov and L.~De~Lathauwer.
 \newblock Generic uniqueness conditions for the canonical polyadic
   decomposition and {INDSCAL}.
 \newblock {\em SIAM J. Matrix Anal. Appl.}, 36(4):1567--1589, 2015.
 
 \bibitem{GuoMironBrieStegeman}
 X.~Guo, S.~Miron, D.~Brie, and A.~Stegeman.
 \newblock {U}ni-{M}ode and {P}artial {U}niqueness {C}onditions for
   {CANDECOMP}/{PARAFAC} of {T}hree-{W}ay {A}rrays with {L}inearly {D}ependent
   {L}oadings.
 \newblock {\em SIAM J. Matrix Anal. Appl.}, 33:111--129, 2012.
 
 \bibitem{Harshman1970}
 R.~A. Harshman.
 \newblock Foundations of the {PARAFAC} procedure: {M}odels and conditions for
   an “explanatory” multi-modal factor analysis.
 \newblock {\em UCLA Working Papers in Phonetics}, 16:1--84, 1970.
 
 \bibitem{Harshman1972}
 R.~A. Harshman.
 \newblock Determination and {P}roof of {M}inimum {U}niqueness {C}onditions for{
   PARAFAC}1.
 \newblock {\em UCLA Working Papers in Phonetics}, 22:111--117, 1972.
 
 \bibitem{HarshmanLundy1984}
 R.~A. Harshman and M.~E. Lundy.
 \newblock The {PARAFAC} model for three-way factor analysis and
   multidimensional scaling.
 \newblock {\em Research methods for multimode data analysis}, pages 122--215,
   1984.
 
 \bibitem{1994HarshmanLundy}
 R.~A. Harshman and M.~E. Lundy.
 \newblock {P}arafac: {P}arallel factor analysis.
 \newblock {\em Comput. Stat. Data Anal.}, pages 39--72, 1994.
 
 \bibitem{Hitchcock}
 F.~L. Hitchcock.
 \newblock The expression of a tensor or a polyadic as a sum of products.
 \newblock {\em J. Math. Phys.}, 6:164--189, 1927.
 
 \bibitem{JiangSid2004}
 T.~Jiang and N.~D. Sidiropoulos.
 \newblock {K}ruskal's permutation lemma and the identification of
   {C}andecomp/{P}arafac and bilinear models with constant modulus constraints.
 \newblock {\em IEEE Trans. Signal Process.}, 52(9):2625--2636, September 2004.
 
 \bibitem{KoldaReview}
 T.~G. Kolda and B.~W. Bader.
 \newblock {T}ensor decompositions and applications.
 \newblock {\em SIAM Rev.}, 51:455--500, 2009.
 
 \bibitem{Kruskal1977}
 J.~B. Kruskal.
 \newblock Three-way arrays: rank and uniqueness of trilinear decompositions,
   with application to arithmetic complexity and statistics.
 \newblock {\em Linear Algebra Appl.}, 18:95--138, 1977.
 
 \bibitem{Leurgans1993}
 S.~E. Leurgans, R.~T. Ross, and R.~B. Abel.
 \newblock A decomposition for three-way arrays.
 \newblock {\em SIAM J. Matrix Anal. Appl.}, 14(4):1064--1083, October 1993.
 
 \bibitem{1988Topographic}
 J.~M$\ddot{\text{o}}$cks.
 \newblock Topographic components model for event-related potentials and some
   biophysical considerations.
 \newblock {\em IEEE Trans. Biomed. Eng.}, 35:482--484, 1988.
 
 \bibitem{Lieven-Nikos_overview}
 N.~D. Sidiropoulos, L.~De~Lathauwer, X.~Fu, K.~Huang, E.~E. Papalexakis, and
   C.~Faloutsos.
 \newblock Tensor decomposition for signal processing and machine learning.
 \newblock {\em Tech. Report 16-34, ESAT-STADIUS, KU Leuven (Leuven, Belgium),
   2016}.
 
 \bibitem{Sorber}
 L.~Sorber, M.~Van~Barel, and L.~De~Lathauwer.
 \newblock Optimization-based algorithms for tensor decompositions: canonical
   polyadic decomposition, decomposition in rank-(${L}_r$,${L}_r$,1) terms and a
   new generalization.
 \newblock {\em SIAM J. Optim.}, 23:695–--720, 2013.
 
 \bibitem{Zhang20131918}
 L.~Zhang, T.-Z. Huang, Q.-F. Zhu, and L.~Feng.
 \newblock Uni-mode uniqueness conditions for {CANDECOMP/PARAFAC} decomposition
   of n-way arrays with linearly dependent loadings.
 \newblock {\em Linear Algebra Appl.}, 439(7):1918--1928, 2013.
 
 \end{thebibliography}

 \begin{itemize}
\item[Address] 
\item[-] Group Science, Engineering and Technology, KU Leuven - Kulak,
 E. Sabbelaan 53, 8500 Kortrijk, Belgium
\item[-] Dept. of Electrical Engineering  ESAT/STADIUS KU Leuven,
 Kasteelpark Arenberg 10, bus 2446, B-3001 Leuven-Heverlee, Belgium
\item[-]iMinds Medical IT
\item[email:]ignat.domanov@kuleuven.be, lieven.delathauwer@kuleuven.be
\end{itemize} 

\end{document}